\documentclass[a4paper,11pt]{article}
\usepackage{fullpage,hyperref}
\usepackage[british]{babel}
\usepackage[comma, square, numbers]{natbib}
\usepackage[utf8]{inputenc}
\usepackage{amssymb}
\usepackage{amsthm}
\usepackage{graphics}
\usepackage{amsmath}
\usepackage{dirtytalk}
\usepackage{amstext}
\usepackage{engrec}
\usepackage[font={footnotesize}]{caption}
\usepackage{framed}
\usepackage{authblk}

\usepackage{mathtools}

\DeclarePairedDelimiter\floor{\lfloor}{\rfloor}
\DeclareMathOperator\var{var}

\usepackage[titletoc,title]{appendix}

\newcommand{\mc}{\mathcal}
\newcommand{\mb}{\mathbb}

\newcommand{\mg}{\mathbf}
\newcommand{\R}{\mb R}

\newcommand{\N}{\mb N}

\newcommand{\PP}{\mb P}

\newcommand{\eea}{\end{align}}

\renewcommand{\epsilon}{\varepsilon}
\renewcommand{\bar}{\overline}

\newcommand{\bo}{\boldsymbol}
\renewcommand{\phi}{\varphi}

\renewcommand\upsilon{\theta}
\usepackage[normalem]{ulem}

\newtheorem{theorem}{Theorem}[section]

\newtheorem{corollary}[theorem]{Corollary}
\newtheorem{lemma}[theorem]{Lemma}
\newtheorem{proposition}[theorem]{Proposition}
\theoremstyle{definition}
\newtheorem{definition}[theorem]{Definition}

\theoremstyle{remark}
\newtheorem{remark}[theorem]{Remark}

\usepackage{xcolor}

\newtheorem{algorithm}{Algorithm}

\newcommand{\balgorithm}{\begin{algorithm}\begin{framed}\ }
\newcommand{\ealgorithm}{\end{framed}\end{algorithm}}

\newcommand{\pippo}{\begin{code} \  \begin{lstlisting}[frame=single]}
\newcommand{\pappo}{\end{lstlisting} \end{code}}

\newcommand{\bd}{\begin{definition}}
\newcommand{\ed}{\end{definition}}

\newcommand{\bt}{\begin{theorem}}
\newcommand{\et}{\end{theorem}}
\newcommand{\bp}{\begin{proposition}}
\newcommand{\ep}{\end{proposition}}
\newcommand{\bc}{\begin{corollary}}
\newcommand{\ec}{\end{corollary}} 
\newcommand{\bl}{\begin{lemma}}
\newcommand{\el}{\end{lemma}}
\newcommand{\br}{\begin{remark}}
\newcommand{\er}{\end{remark}}

\newcommand{\Lip}{\mathrm{Lip}}
\DeclareMathOperator{\Id}{Id}

\DeclareMathOperator{\Dist}{Dist}

\title{Random Composition of L-S-V Maps Sampled Over Large Parameter Ranges}
\date{}
\author[1]{Christopher Bose}
\author[1]{Anthony Quas}
\author[2]{Matteo Tanzi}
\affil[1]{Department of Mathematics and Statistics,
University of Victoria,
PO BOX 1700 STN CSC,
Victoria, B.C.,
Canada V8W 2Y2}
\affil[2]{Courant Institute of Mathematical Sciences, New York University, New York, NY 10012, USA}

\begin{document}
\maketitle

\begin{abstract}
Liverani-Saussol-Vaienti (L-S-V) maps form a family of piecewise differentiable dynamical 
systems on $[0,1]$ depending on one parameter $\omega\in\R^+$. These maps  
are everywhere expanding apart from a neutral fixed point. It is well known that 
depending on the amount of expansion close to the neutral point, they 
have either an absolutely continuous invariant probability measure and polynomial decay of 
correlations ($\omega <1$), or a unique physical measure that is singular and 
concentrated at the neutral point ($\omega >1$).  In this paper, we study the 
composition of L-S-V maps whose parameters are randomly sampled from a 
range in $\R^+$, and where these two contrasting behaviours are 
mixed. We show that if the parameters $\omega<1$ are sampled with positive 
probability, then the stationary measure of the random system is absolutely 
continuous; the annealed decay rate of correlations is close (or in some cases equal) 
to the fastest rate of decay among those of the sampled systems; and suitably 
rescaled Birkhoff averages converge to limit laws. In contrast to previous studies where $\omega \in [0,1]$, we allow $ \omega >1$ in our sampling distribution.  We also show that one can obtain similar decay of correlation rates for $\omega \in \R^+$,  when sampling is done with respect to a family of smooth, heavy-tailed distributions.
\end{abstract}
\section{Introduction}

Let us consider the one-parameter family of L-S-V maps \cite{liverani1999probabilistic} 
$\{f_\omega\}_{\omega\in\R^+}$,  $f_\omega:[0,1]\rightarrow[0,1]$, defined as
\begin{equation}\label{Eq:LSV}
f_\omega(x)=\left\{\begin{array}{cr}
x(1+2^\omega x^\omega)&x\in[0,\tfrac 12)\\
2x-1& x\in[\tfrac 12,1]
\end{array}\right.
\end{equation}
These maps were introduced as a simplified version of the previously studied Pomeau-Manneville family
$x \rightarrow x + x^{1+\omega}(\textnormal{mod} \,1)$ \cite{pomeau-manneville}, retaining the essential property of having a neutral fixed point at $x=0$ but all $f_\omega$ having the same two intervals of monotonicity and uniformly expanding, affine, right-hand branch\footnote{It is reasonable to expect that the results one can obtain for L-S-V maps hold for the original P-M maps.  However we take advantage of the affine second branch in L-S-V at various points in the proofs so extending our arguments to the non-affine case is likely to be technically complicated}.  Since then,  L-S-V maps have become a standard example of dynamical systems with 
intermittency, alternating stretches of time where orbits exhibit chaotic behaviour, 
and long stretches of time where they are almost constant and close to zero. 
The uniform expansion of the maps away from the fixed point $x=0$ is responsible 
for the chaotic behaviour, while the fact that $f'_\omega(0)=1$ implies that it takes
a long time before orbits can escape the vicinity of $x=0$. It is well known that for 
$\omega\in (0,1)$, $f_\omega$ has a unique physical probability measure which is absolutely 
continuous and  exhibits polynomial decay of correlations where the rate of  
decay depends on $\omega$ \cite{Young}. For $\omega\ge 1$, $f_\omega$ does 
not have any invariant absolutely continuous  probability measure, but  it has a unique invariant 
absolutely continuous \emph{infinite} measure, and the Dirac delta concentrated at zero
is the unique probability physical measure. This is due to the fact that for $\omega\ge 1$, 
the repulsion in a neighbourhood of zero is so small that, asymptotically, almost every 
orbit spends most of its time arbitrarily close to zero.  

After fixing a compact interval $[\alpha,\beta]\subset \R^+$ with $\alpha\leq\beta$, 
we study the composition of maps sampled randomly with respect to a given probability 
measure $\nu$ on $[\alpha,\beta]$, and characterise the average statistical properties of their orbits 
(\emph{annealed} results). The case where $[\alpha,\beta]\subset (0,1]$ has been 
previously considered. In  \cite{BoseBahsoun} and \cite{bahsoun2016mixing} the 
authors show that if the distribution is discrete and samples only finitely many values 
in $[\alpha,\beta]$, the annealed correlations decay at a rate equal to that of the 
system with the smallest $\omega\in[\alpha,\beta]$ which is sampled with positive 
probability. Other works deal with \emph{quenched} results that, in contrast with 
annealed results, establish decay of correlations and convergence to limit laws for 
almost every sequence of maps sampled with respect to the given measure on 
$[\alpha,\beta]\subset (0,1]$ \cite{bahsoun2017quenched}  (see also \cite{baladi2002almost} for  quenched results on random composition of unimodal maps). In \cite{aimino2014polynomial}, \cite{nicol2018central}, and  \cite{nicol2019large}
an arbitrary composition of maps (sequential random systems) from $[\alpha,\beta]\subset (0,1)$, where $\alpha$ and $\beta$ satisfy some additional technical conditions, is shown to give decay of correlations at a 
speed bounded above by that of the system $\beta$,   satisfy central limit theorems, and large deviations principles.   

In contrast with the above cited works, we  focus on the case where $\alpha<1<\beta$ 
and thus maps with both $\omega<1$ and $\omega>1$ are composed. In this case, there 
is competition between two contrasting behaviours: maps with $\omega<1$ tend to 
spread mass over the whole space, while maps with $\omega>1$, although still 
expanding on most of the space, tend to accumulate mass at the neutral fixed point. 
We show that it is enough for maps with $\omega<1$ to be sampled 
with positive probability, to ensure that the average random system has an absolutely continuous 
stationary probability measure, polynomial decay of  annealed correlations, and convergence to limit laws. 

The case of discrete $\nu$ (i.e. $\nu=\sum p_i\delta_{\omega_i}$) was
previously treated in \cite{bahsoun2016mixing} for $\beta\leq 1$ using the 
Young tower approach \cite{Young} to find estimates on the decay of correlations. 
In this analysis, the base of the skew product is conjugated to a piecewise affine and
uniformly expanding system. This allows one, after inducing  on a suitable subset of the 
phase space, to reduce the analysis to the study of a Gibbs-Markov system \cite{aaronson2001local} 
with countably many invertible branches. This construction is not possible in the 
case of a non-discrete $\nu$ where one can find uncountably many inverse branches 
all defined on measurable sets of zero measure that cover a set of positive measure, 
thus obstructing the construction of a countable Gibbs-Markov structure.

Our main approach in this paper is based on the renewal theory of operators  which was introduced by Sarig
\cite{sarig2002subexponential}  and further developed by Gou\"ezel \cite{gouezel2004sharp}), and which  can deal with any 
measure $\nu$ on $[\alpha,\beta]\subset \R^+$. For the machinery to work, it is crucial  to bound  the 
distortion of the composition of different maps with parameters chosen 
arbitrarily from $[\alpha,\beta]$.  For a single (deterministic) map this was done by Young \cite{Young} using a direct calculation, valid for all $\omega>0$.    For random maps, distortion estimates have been obtained  
in the case $\beta\le 1$ using the Koebe principle and 
non-positive Schwartzian derivative  as in  \cite{Gouezel_random_skew} and \cite{BoseBahsoun}. However, this technique for random maps fails when $\beta >1$ since $f_\omega$ has points where the Schwartzian derivative is positive when $\omega >1$. In the following we give a  direct estimate 
of the distortion, more in the spirit of \cite{Young}, that encompasses the general case 
$[\alpha,\beta]\subset \R^+$,  including $\beta>1$.  

 In the case where the interval of parameters is unbounded, i.e. $\beta=\infty$, the  approach above does not work as a uniform bound on the distortion is missing. However, in the special situation where $\nu$ is an absolutely continuous measure on $[\alpha,+\infty)$ with very regular density, we 
present a second approach that exploits the continuous distribution in parameter 
space by looking at the system as a mixture between a diffusion process, and a 
deterministic uniformly expanding system.  The estimates on the (annealed) decay of 
correlations are obtained using the theory of Markov chains with subgeometric 
rates of convergence \cite{tuominen1994subgeometric}.  We are going to present this approach in the case where the distribution of $\nu$ has fat polynomial tails at infinity, meaning that large parameters are sampled with high probability.

\section{Setting and Results}
Let us consider the one-parameter family of L-S-V maps
$\{f_\omega\}_{\omega\in\R^+}$ as in \eqref{Eq:LSV} above.  Given
$\alpha,\beta\in\R^+$ with $\alpha<1$ and $\beta>\alpha$, we consider a
probability measure $\nu$ on the compact interval $[\alpha,\beta]$,
and $\nu^{\N_0}$ the product measure on
$\Omega:=[\alpha,\beta]^{\N_0}$. We assume without loss of generality 
that $\alpha$ belongs to the topological support of $\nu$ (that is
$\nu([\alpha,\alpha+\delta])>0$ for any $\delta>0$).

These data define the random
dynamical system taking skew-product form $F:\Omega\times
[0,1]\rightarrow \Omega\times [0,1]$
\begin{equation}\label{Eq:ContSkew}
F(\omega,x)=(\sigma\bo\omega, f_{\omega_0}x),\quad 
\bo\omega=(\omega_0\omega_1...)\in\Omega, \;x\in[0,1]
\end{equation}
where the reference measure on $\Omega\times [0,1]$ is $\mb
P=\nu^{\N_0}\otimes m$, with $m$ the Lebesgue measure on
$[0,1]$. Given $\bo\omega\in\Omega$ and $n\in\N$, we denote by
$f_{\bo\omega}^n=f_{\omega_{n-1}}\circ ...\circ f_{\omega_0}$.  We are
going to describe the annealed statistical properties of this random
system, i.e. the statistical properties of $F$ averaged with respect
to the reference measure $\nu^{\N_0}$.  In the following, for a measurable
$A\subset\Omega\times[0,1]$ we denote by 
\[\tau_A(\bo
\omega,x):=\inf\{n\in\N|\; F^n(\bo \omega,x)\in A \},
\] and for a
measurable $J\subset [0,1]$ we denote with an abuse of notation
\[\tau_J(\bo \omega,x):=\inf\{n\in\N|\; F^n(\bo \omega,x)\in
\Omega\times J \}.\]
 
\paragraph{Decay of correlations.}
We show that the random system has
an absolutely continuous stationary probability measure and exhibits decay of
correlations at polynomial speed.

\begin{theorem}\label{Thm:DecCorr}
Let $0<\alpha<1$, $\alpha\le \beta$ and let $\nu$ be a probability measure on
$[\alpha,\beta]$ such that $\alpha$ lies in the topological support of $\nu$. 
Then there is an absolutely continuous stationary
measure $\pi$, i.e. $\nu^{\N_0}\otimes \pi$ is invariant under $F$.
For $\psi\in L^\infty([0,1],m)$, $\phi\in \Lip([0,1],\R)$ and all $1<\gamma<\frac 1\alpha$, we have
\[
\left|\int \psi\circ F^n(\bo\omega,x)\phi(x) d\mb P(\bo\omega,x)-\int
\phi(x)dm(x)\int\psi(x)d\pi(x)\right|\leq
\mc O\left(\frac{1}{n^{\gamma-1}}\right).
\]
\end{theorem}
It is natural here to compute correlation integrals with respect to the reference measure $\mb P$ since one does not, in general, know $\pi$ explicitly. 
Note that $\int 
\phi(x)dm(x)= \int \phi (x) d\mb P(\bo\omega,x)$ so, heuristically, we interpret this result as weak convergence of the sequence 
$\psi\circ F^n$ to $\int\psi(x)d\pi(x)$, measured against Lipschitz test functions. 
An elementary calculation shows $\mb P$ (and $m$) above can be replaced by the stationary measure 
$\nu^{\N_0}\otimes \pi$ in both integrals, obtaining the same decay rates with respect to the stationary measure, more in line with results in classical probability.

We prove this theorem using the renewal theory for
operators (see Section \ref{Sec:DecayRen}). This approach is more
standard in the study of dynamical systems with nonuniform hyperbolic
properties, and relies on the bound on distortion given in Section
\ref{Sec:BndDist}. However, we have to deal with the fact that the induced map we are going to study,
$F_Y:= F^{\tau_Y}$ with $Y:=\Omega\times[1/2,1]$, is not Gibbs-Markov, as is usually assumed to be
the case. In fact, there is an uncountable partition of sets with zero
measure that are mapped bijectively onto $Y$ by $F_Y$ and that cover a
set of positive measure. Therefore, we have to prove the spectral
properties of the operators involved directly.

In Section \ref{sec:diffus} we consider the case where $\nu$ is absolutely
continuous with a power law distribution, and prove a decay of correlations statement
for bounded measurable functions. We look at the diffusion process induced by the
skew product map $F$ on the vertical fibre $[0,1]$ and use the
theory of Markov chains with subgeometric rates of convergence to
their stationary state. This approach does not use the bound on the
distortion, whose role in the arguments is played by the diffusion
and gives decay of correlations for arbitrary $L^\infty$ functions.

For $0<\alpha<1$ and $\epsilon>0$, let $\nu_{\alpha,\epsilon}$ denote the measure
on $[\alpha,\infty)$ defined by 
$$
\nu_{\alpha,\epsilon}(A)=\int_A \frac{\epsilon\alpha^\epsilon}{t^{\epsilon+1}}\,dt,
$$
for any $A\subset [\alpha,\infty)$. Alternatively, $\nu_{\alpha,\epsilon}$ is characterized by
$\nu_{\alpha,\epsilon}(t,\infty)=(\frac t\alpha)^{-\epsilon}$ for any $t\ge \alpha$.

\begin{theorem}[$L^\infty$ decay of correlations for power law parameter distributions]
\label{thm:Linftydecorr}
Let $0<\alpha<1$ and let $\epsilon>0$ and consider the random composition of LSV maps where
the parameters are chosen i.i.d.~from $\nu_{\alpha,\epsilon}$. The corresponding Markov chain has
a stationary probability distribution, $\pi$. 
Let $\phi,\psi\in L^\infty[0,1]$. Then for all $1<\gamma<\frac1\alpha$,
$$
\left|
\int \int \psi\circ F^n(\bo\omega,x)\phi(x) \,dm(x)\,d\nu_{\alpha,\epsilon}^{\N}(\bo\omega)-
\int \phi(x)\,dm(x) \int \psi(y)\,d\pi(y)\right|=\textnormal{o}\left(\frac{1}{n^{\gamma-1}}\right).
$$
\end{theorem}

\paragraph{Convergence to limit laws.} 
Let $\phi:[0,1]\rightarrow \R$ be a Lipschitz function.
Define $\phi_Y:\Omega\times[1/2,1]\rightarrow \R$ by  
\[
\phi_Y(\bo\omega,x):=\sum_{i=0}^{\tau_Y(\bo\omega,x)-1}\phi\circ \pi_2\circ F^i(\bo\omega,x),
\]  
%Call $G$ the distribution of $\phi_Y$ (with respect to the reference
%measure), that is $G(t)=\mathbb P(\{(\bo\omega,x)\colon \phi_Y(\bo\omega,x)\le t\}$. 
%We observe that $\phi_Y$ is bounded below: to see this, 
%notice that $\phi$ is positive on a neighbourhood $(0,a)$ of 0 and 
%there exists an $n_0$ such that $f_\beta^{n_0}(a)\in[\frac 12,1]$. 
%Hence, for any $(\bo \omega,x)\in Y$, the orbit may spend at most $n_0$ 
%steps in $\Omega\times [a,\frac12)$ before returning to $Y$.
%
%The tails of $G$ depend primarily on the distribution of the return time
%$\tau_Y$. Assume that $G$ is in the basin of a stable law with
%exponent $p\in(0,2]$, i.e. there is a slowly varying function $L:\R^+\rightarrow \R^+$ 
%such that $x^p(1-G(x))=L(x)$. 
 where $\pi_2$ denotes projection to the second coordinate.  

We denote the Birkhoff sums with respect to $F_Y$ by
\[
S^n\psi:=\psi+\psi\circ F_Y+...+\psi\circ F_Y^{n-1}
\]
\begin{theorem}\label{Thm:LimitUnfold}
Let $0<\alpha<1$, $\alpha\neq 1/2$, and $\alpha<\beta$.
Let $\nu$ be a probability measure supported on $[\alpha,\beta]$ with $\alpha$ in the topological 
support of $\nu$. 
Let $\phi$ be a Lipschitz function on $[0,1]$ with $\phi(0)\neq0$.
Then there is a stable law $\mc Z$  a sequence $A'_n$, and a sequence $B'_n$ such that
\[
\lim_{n\rightarrow\infty}\frac{S^n\phi_Y-A'_n}{B'_n}\rightarrow\mc Z
\]
where the convergence is in distribution. Furthermore
\begin{itemize}
\item[i)] if $\alpha<1/2$, then $\mc Z$ is a Gaussian;
\item[ii)] if $\alpha>1/2$, then $\mc Z$ is a stable law with index $1/\alpha$.
\end{itemize}

\end{theorem}
These limit laws assuming a compact parameter range are proved using the Nagaev-Guivarc'h approach as
outlined in  \cite{gouezel2004central} and  \cite{Gouezel2}. The spectral properties of the transfer
operator of $F_Y$ will play a crucial role to this end. The case $\alpha=\frac 12$ is not addressed in this work, and is likely to be delicate since even for the deterministic map $f_{\omega = \frac 12}$ the analysis in
\cite{gouezel2004central} (Section 1.3) derives a Gaussian limit under normalization $B_n'\sim \sqrt{n \log n}$ whereas for $\omega > 1/2$ the limit is a stable law similar to the above.  One can therefore expect that when $\alpha = \frac 12$, precise properties of the distribution of $\nu$ near $\frac 12$ will be needed to derive limit theorems in the random case, in contrast to the simpler results stated above. 

For $\phi(0)> 0$ our argument exploits the approximation\footnote{For a more precise statement, see the proof of Theorem 2.3 in Section 5.}  $\phi_Y \approx \tau_Y \cdot \phi(0)$ in distribution.  A similar argument holds when $\phi(0) <0$.   When  $\varphi(0)=0$ the estimate on $\phi_Y$ is more delicate, even in the deterministic case, so we do not consider that case in the current work. 

Finally, we do not consider limit theorems in the case of unbounded parameter range and heavy tails, such as $\nu_{\alpha,\epsilon}$.  Although there may be path to obtain this via the Nagaev-Guivarc'h approach, the details will require a complete reworking of the bounded range case. We suggest this for possible investigation in the future.

\section{Bound on the distortion}\label{Sec:BndDist}

\begin{definition}	
Suppose $f:[0,1]\rightarrow[0,1]$ is a piecewise differentiable map
with $f'> 0$ on any $J\subset[0,1]$ such that $f|_J$ is
differentiable. The distortion of $f$ on $J$ is defined as
\[
\Dist(f|J)=\sup_{x,y\in J}\log\frac{|D_xf|}{|D_yf|}.
\]
\end{definition}
We restrict our attention to $J\subset[0,1/2)$ and for the moment write (abusing notation),
\[
f_\omega^{-1}=(f_{\omega}|_{[0,1/2)})^{-1}.
\]

\begin{proposition}\label{Prop:UnifBndDist}
There is $K>0$ such that for any $\omega\in\Omega=[\alpha,\beta]^{\N_0}$, $I' \subset
[1/2,1]$, and $n\in\N$
\[
\Dist(f^n_{\bo \omega}|(f_{\bo \omega}^n)^{-1}(I'))\leq K\log\frac{\sup I'}{\inf I'}.
\]
\end{proposition}  
As an immediate corollary to the previous proposition we obtain:
\begin{corollary}\label{Cor:Dist}
There is $K'>0$ such that for any $\bo \omega\in\Omega$ and interval
$[x,y]\subset[0,1)$ with $f_{\bo \omega}^n$, mapping $[x,y]$
bijectively to $f_{\bo \omega}^n([x,y])\subset [1/2,1)$
  \[
  \Dist(f^n_{\bo \omega}|[x,y])\leq K'|f_{\bo \omega}^n(x)-f_{\bo \omega}^n(y)|.
  \]
\end{corollary}
\begin{proof}
  There is $k\in\N$ and numbers $0=n_0<n_1<n_2<...<n_k=n$ where $n_i$
  is the $i$-th return of all points in $[x,y]$ to the set $[1/2,1]$. From Proposition
  \ref{Prop:UnifBndDist}, we get that for every $0\le i\le k-1$
\[
\Dist(f^{n_{i}-n_{i-1}}_{\sigma^{n_{i-1}}\bo \omega}|f^{n_{i-1}}_{\bo \omega}([x,y]))\leq K''|f^{n_{i}}_{\bo\omega}(x)-f^{n_{i}}_{\bo\omega}(y)|.
\]
Since $|f^{n_{i}}_{\bo\omega}(x)-f^{n_{i}}_{\bo\omega}(y)|\leq
2^{-(k-i)}|f^{n}_{\bo\omega}(x)-f^{n}_{\bo\omega}(y)|$ one gets
\begin{align*}
  \Dist(f^n_{\bo \omega}|[x,y])&\leq
  \sum_{i=1}^{k}\Dist(f^{n_{i}-n_{i-1}}_{\sigma^{n_{i-1}}\bo
    \omega}|f^{n_{i-1}}_{\bo \omega}([x,y]))\\ &\leq
  K''\sum_{i=1}^{k}2^{-(k-i)}|f^{n}_{\bo\omega}(x)-f^{n}_{\bo\omega}(y)|\\ &\leq
  2K''|f^{n}_{\bo\omega}(x)-f^{n}_{\bo\omega}(y)|.
\end{align*}
\end{proof}
To prove the proposition we use the following lemma.
 
\begin{lemma}
For any closed interval $J\subset(0,1]$ define 
\[
\mc D(J)=(1+\beta)\log\frac{\sup J}{\inf J}.
\]
Then,
\begin{equation}\label{Eq:Lem1}
 \max_{\omega\in[\alpha,\beta]}\left\{\mc
 D(f_\omega^{-1}(J))+\Dist(f_\omega|f_\omega^{-1}(J))\right\}\le \mc
 D(J).
\end{equation}
\end{lemma}
\begin{proof}
Fix $J=[a,b]$. For  $\omega\in[\alpha,\beta]$
\begin{align*}
\mc D(f_\omega^{-1}(J))+\Dist(f_\omega|f_\omega^{-1}(J))
&=(1+\beta)\log\frac{ f_\omega^{-1}(b)}{ f_\omega^{-1}(a)}+
\log\frac{D_{f_\omega^{-1}(b)}f_\omega}{D_{f_\omega^{-1}(a)}f_\omega}
\end{align*}
where in the estimates above we used the fact that $(f_\omega)'$ is
positive and monotonically increasing.  Calling
$a_{-1}:=f_\omega^{-1}(a)$ and $b_{-1}:=f_\omega^{-1}(b)$ all we need
to show is that
\begin{equation}\label{Eq:Ineq1}
(1+\beta)\log\frac{ b_{-1}}{
    a_{-1}}+\log\frac{D_{b_{-1}}f_\omega}{D_{a_{-1}}f_\omega}\leq
  (1+\beta)\log\frac{f_\omega(b_{-1})}{f_\omega(a_{-1})}.
\end{equation}
The left-hand side equals
\begin{align*}
\mbox{LHS}=\log\left(\left(\frac{b_{-1}}{a_{-1}}\right)^{(1+\beta)}\frac{1+(\omega+1)2^\omega
  b_{-1}^\omega}{1+(\omega+1)2^\omega a_{-1}^\omega}\right)
\end{align*}
and the right-hand side equals
\begin{align*}
\mbox{RHS}=\log\left(\left(\frac{b_{-1}}{a_{-1}}\right)^{(1+\beta)}\left(\frac{1+2^\omega
  b_{-1}^\omega}{1+2^\omega a_{-1}^\omega}\right)^{(1+\beta)}\right).
\end{align*}
\begin{align*}
\mbox{LHS$-$RHS}&= \log\left(\frac{1+(\omega+1)2^\omega
  b_{-1}^\omega}{1+(\omega+1)2^\omega
  a_{-1}^\omega}\left(\frac{1+2^\omega a_{-1}^\omega}{1+2^\omega
  b_{-1}^\omega}\right)^{(1+\beta)}\right)\\ 
  &=\log\left(\frac{(1+2^\omega
  a_{-1}^\omega)^{(1+\beta)}}{(1+(\omega+1)2^\omega
  a_{-1}^\omega)}\frac{(1+(\omega+1)2^\omega b_{-1}^\omega)}{(1+ 2^\omega
  b_{-1}^\omega)^{(1+\beta)}}\right)
\end{align*}
The fact that LHS$-$RHS$\le0$  follows from the fact that 
\[
g(x)=\frac{(1+x)^{(1+\beta)}}{1+(\omega+1)x}. 
\]
is non-decreasing as an elementary derivative calculation shows,
%\footnote{\begin{align*}
%\frac{d}{dx}\frac{(1+x)^{(1+\beta)}}{1+(\omega+1)x}
%&=\frac{(1+\beta)(1+x)^{\beta}(1+(\omega+1)x)-
%(\omega+1)(1+x)^{(1+\beta)}}{[1+(\omega+1)x]^2}\\
%&=\frac{(1+x)^{\beta}}{[1+(\omega+1)x]^2}[\beta-\omega+(\omega+1)(\beta)x]>0\quad
%$\mbox{ for }x>0.
%\end{align*}}
and the fact that LHS$-$RHS$=\log\frac{g(2^\omega a_{-1}^\omega)}
{g(2^\omega b_{-1}^\omega)}$ with
$b_{-1}\ge a_{-1}$.  This concludes the proof of \eqref{Eq:Ineq1}.
\end{proof}
	
\begin{proof}[Proof of Proposition \ref{Prop:UnifBndDist}] 
Fix $\bo \omega\in\Omega$ and $n\in\N$ and let
$I'$ be a sub-interval of $[\frac 12,1]$. Define 
$J_0=f_{\omega_0}^{-1}\circ\ldots\circ f_{\omega_{n-1}}^{-1}(I')$ and for every $1\leq i\leq n-1$ and
\[
J_{i}:=f_{\bo \omega}^i(J_0)=f_{\omega_{i}}^{-1}\circ...\circ
f_{\omega_{n-1}}^{-1}(I').
\]
From the definition of distortion and monotonicity of $f_\omega$ and
$(f_\omega)'$
\[
\Dist(f_{\bo\omega}^n|J_0)=\sum_{i=0}^{n-1}\Dist(f_{\omega_{i}}|J_{i}).
\]
Repeatedly applying \eqref{Eq:Lem1} 
\begin{align*}
\mc D(I')&\ge \mc D(J_{n-1})+\Dist(f_{\omega_{n-1}}|J_{n-1})\ge
\mc D(J_{0})+\sum_{i=0}^{n-1}\Dist(f_{\omega_{i}}|J_{i})
\end{align*}
which implies that
\[
\sum_{i=0}^{n-1}\Dist(f_{\omega_{i}}|J_{i})\le \mc
D(I')=(1+\beta)\log\frac{\sup I'}{\inf I'}.
\]
\end{proof}

\section{Estimates on the Return Times to the Inducing Set}\label{Sec:EstReturn}
Let $Y=\Omega\times[1/2,1]$ and $\mb P_Y$ the probability measure obtained by
restricting and then normalizing $\mb P$ to $Y$. We use the following condition on the tails of the 
return time to $Y$ in the renewal theory approach
\begin{equation}\label{Eq:Assum1}\tag{{\bf C1}}
\exists \gamma>1\mbox{ such that } \mb P_Y\left(\{(\bo \omega,x)\in
Y:\;\tau_Y(\bo\omega,x)>n\}\right)=\mc O(n^{-\gamma}).
\end{equation}
In the diffusion driven case we are going to use the following,
closely related condition: there is a non-decreasing function
$R:\N\rightarrow \R^+_0$ such that
\begin{equation}\label{Eq:Assum2}\tag{{\bf C2}}
\sum_{n=1}^\infty \mb P_Y\left(\{(\bo \omega,x)\in
Y:\;\tau_Y(\bo\omega,x)=n\}\right)\cdot R(n)<\infty.
\end{equation}

In fact, for $t>1$, if condition \eqref{Eq:Assum1} holds for all $\gamma<t$
then condition \eqref{Eq:Assum2} holds for $R(n)=n^\gamma$ for all $\gamma<t$.
To see this, we use summation by parts, showing
\begin{align*}
\sum_{n=1}^\infty \mb P(\tau=n)R(n)&=
\sum_{n=1}^\infty (\mb P(\tau>n-1)-\mb P(\tau>n))R(n)\\
&=R(1)+\sum_{n=1}^\infty \mb P(\tau>n)(R(n+1)-R(n))
\end{align*}

Notice that, having fixed the family $\{f_\omega\}_{\omega\in
  [\alpha,\beta]}$, \eqref{Eq:Assum1} and \eqref{Eq:Assum2} are
conditions on the measure $\nu$. Sharp bounds for the expressions
in \eqref{Eq:Assum1} and \eqref{Eq:Assum2} have been obtained for
various $\nu$ in \cite{bahsoun2017quenched}. Below we prove the following proposition
\begin{proposition}\label{Prop:As1and2valid}
Assume that $\nu$ is a probability measure supported on $[\alpha,\beta]$ with $0<\alpha<1$ lying
in the topological support of $\nu$. Then conditions
\eqref{Eq:Assum1} and \eqref{Eq:Assum2} hold. In particular
\begin{enumerate}
\item For any $1<\gamma<\frac{1}{\alpha}$,
Condition \eqref{Eq:Assum1} holds and Condition \eqref{Eq:Assum2} holds with 
$R(n)=n^{\gamma}$;
\item if $\nu(\{\alpha\})>0$,
%there is $\lambda\in(0,1]$ and $\tilde \nu$ a probability measure such that 
%$\nu=\lambda \delta_{\alpha}+(1-\lambda)\tilde\nu$
then Condition \eqref{Eq:Assum1} holds with $\gamma=\frac{1}{\alpha}$
and Condition \eqref{Eq:Assum2} holds with $R(n)=n^{1/\alpha}$.

\end{enumerate}
\end{proposition}

Conditions \eqref{Eq:Assum1} and \eqref{Eq:Assum2}
concern the return times of orbits to the inducing set under the map
$F$. They can be verified on a case-by-case basis in the following
way (see also \cite{bahsoun2017quenched}).  Call $f_{\omega,1}:=
f_\omega|_{[0,1/2)}$ the first invertible branch of $f_\omega$, and
$g:=2x-1$ mod 1 on $[1/2,1]$ the second branch common to all the
maps in the family. Define
$x_{n}(\bo\omega):=f^{-1}_{\omega_{1}}...f^{-1}_{\omega_{n-1}}(1/2)$. Then
$\{(\bo \omega,x)\in Y:\;\tau_Y(\bo\omega,x)>n\}=\{(\bo \omega,y):
y\in[1/2,g^{-1}(x_n(\bo\omega))]\}.$ 

This implies that 
\begin{equation}\label{Eq:EstTailstauY}
\mb P_Y((\bo \omega,x)\in Y\colon \tau_Y(\bo\omega,x)>n)=\tfrac 12\mb E [x_n(\bo\omega)],
\end{equation}
%This implies that
%\begin{align*}
%\mb P_Y\left(\{(\bo \omega,x)\in
%Y:\;\tau_Y(\bo\omega,x)=n\}\right)&=\frac{1}{2}\mb
%E[x_{n-1}(\bo\omega)]-\frac{1}{2}\mb
%E[x_n(\bo\omega)]\\ &=\frac{1}{2}\mb
%E[x_{n-1}(\sigma\bo\omega)-x_n(\bo\omega)]
% \end{align*}
where the expectation $\mb E$ is with respect to
$\nu^{\N}$. %Since
% and we used the invariance of the measure $\nu$ under
%the shift. Since
%$f_{\omega_{1}}(x_n(\bo\omega))=x_{n-1}(\sigma\bo\omega)=
%x_n(\bo\omega)+2^{\omega_1}x_n(\bo\omega)^{\omega_1+1}$ one obtains
% \begin{equation}\label{Eq:EstTailstauY}
% \mb P_Y\left(\{(\bo \omega,x)\in
% Y:\;\tau_Y(\bo\omega,x)=n\}\right)=\frac{1}{2}\mb
% E[2^{\omega_1}x_n(\bo\omega)^{\omega_1+1}].
% \end{equation}
The equality above together with computations as in \cite{BoseBahsoun}
allow us to prove Proposition \ref{Prop:As1and2valid}.

\begin{proof}[Proof of Proposition \ref{Prop:As1and2valid}]
From the assumptions, $\alpha:=\min
\mbox{supp}(\nu)>0$. Let $1<\gamma<\frac 1\alpha$,
so that from the definition of topological support, 
$p_0:=\nu([\alpha,\gamma^{-1}])>0$. It is known that if
$\bo\omega$ is a sequence such that
$\sum_{j=1}^n\bo1_{\{\omega_j\in[\alpha,\gamma^{-1}]\}}>M$,
then $x_n(\bo\omega)\leq \floor{M}^{-\gamma}$ (see for example \cite{BoseBahsoun}) and is a consequence of
the monotonicity found in the family of functions
$\{f_\omega\}_{\omega\in\R^+}$. It follows from the Hoeffding concentration
inequality for i.i.d. and bounded random variables that
\[
\nu^{\N}\left(\bo\omega:\;
\frac{1}{n}\sum_{i=1}^{n}\bo 1_{\{\omega_i\in[\alpha,\gamma^{-1}]
  \}}<p_0-\epsilon\right)\leq e^{-2n\epsilon^2}
\]
%which can be rewritten as
%\[
%\nu^{\N}\left(\bo\omega:\;
%\sum_{i=1}^{n}\bo1_{\{\omega_i\in[\alpha,\alpha+\delta']
%  \}}<n(p_0-\epsilon)\right)\leq e^{-2n\epsilon^2}.
%\]
so that calling 
\begin{equation}\label{Eq:SetA}
A_n:=\left\{\bo\omega:\,\sum_{i=1}^{n}\bo 1_{\{\omega_i\in[\alpha,\gamma^{-1})\}}(\bo\omega)<n(p_0-\epsilon)\right\},\end{equation} 
$\nu^\N(A_n)\le e^{-2n\epsilon^2}$.
Now, using
\eqref{Eq:EstTailstauY}:
\begin{align*}
\mb P_Y(\{(\bo\omega,x)\colon \tau_Y(\bo\omega,x)>n\})&=\frac 12\int d\nu^{\N}(\bo\omega)
x_n(\bo\omega)\\
&=\frac 12\left[\int_{A_n}d\nu^{\N}(\bo\omega)+\int_{A_n^c}d\nu^{\N}(\bo\omega)\right]
x_{n}(\bo\omega)\\
&\le \tfrac 12(e^{-2n\epsilon^2}+(n(p_0-\epsilon))^{-\gamma})=\mc O(n^{-\gamma}).
%\int \left[
%e^{-2n\epsilon^2}2^{\omega_1}\left(\frac{1}{2}\right)^{\omega_1+1}+\frac{1}{2}
%\left(\frac{n(p_0-\epsilon)}{2}\right)^{(\omega_1+1)(-\frac{1}{\alpha}+\delta')}\right]
%d\nu(\omega_1)
\end{align*}
establishing Condition \eqref{Eq:Assum1} for any $\gamma<\frac 1\alpha$. As pointed out above, this implies that
Condition \eqref{Eq:Assum2} holds with $R(n)=n^\gamma$ for any $\gamma<\frac 1\alpha$. 

If $\nu(\{\alpha\})>0$, then one can repeat the above reasoning with 
$\gamma=\frac{1}{\alpha}$ to obtain
\[
\mb P_Y(\{(\bo\omega,x)\colon \tau_Y(\bo\omega,x)>n\})\le \mc O(n^{-\frac{1}{\alpha}}).
\]
%\[
%\mb P_Y\left(\{(\bo \omega,x)\in Y:\;\tau_Y(\bo\omega,x)>n\}\right)=
%\sum_{i>n}\mb P_Y\left(\{(\bo \omega,x)\in Y:\;\tau_Y(\bo\omega,x)=i\}\right).
%\]
\end{proof}

\section{Renewal Theory Approach}\label{Sec:Renewal}
 
In this section we prove Theorem \ref{Thm:DecCorr} on correlation decay, and Theorem \ref{Thm:LimitUnfold} on convergence to stable laws. In Section \ref{Sec:DecayRen} we treat the correlation decay using the renewal theory for transfer operators, as introduced by Sarig \cite{sarig2002subexponential} and further developed by Gou\"ezel in \cite{gouezel2004sharp}, while in Section \ref{Sec:ConStabLaws} we use the Nagaev-Guivarc'h approach to prove convergence to limit laws.

\subsection{Decay of Correlations via the Renewal Theory for Transfer Operators}
\label{Sec:DecayRen}
We apply the following general theorem
that can be found in \cite{gouezel2004sharp}. Let us denote the open unit disk by 
$\mb D=\{z\in\mb C:\; |z|<1\}$.

\begin{theorem}[Theorem 1.1 \cite{gouezel2004sharp}]\label{Thm:RenGouez}
Let $\{T_n\}_{n\ge 0}$ be bounded operators on a Banach space $(\mc
B,\|\cdot\|)$ such that $T(z)=\Id+\sum_{n\ge 1}z^nT_n$ converges for
every $z\in \mb D$. Assume that:
\begin{itemize}
\item for every $z\in\mb D$, $T(z)=(\Id-R(z))^{-1}$ where
  $R(z)=\sum_{n\ge 1}z^nR_n$ and $\{R_n\}_{n\ge 1}$ are bounded
  operators on $\mc B$ such that $\sum_{n\ge 1}\|R_n\|<+\infty$;
\item 1 is a simple isolated eigenvalue of $R(1)$;
\item for every $z\in\bar{\mb D}\backslash\{1\}$, $\Id-R(z)$ is invertible. 
\end{itemize}
Let $\Pi$ be the eigenprojection of $R(1)$ at 1.  If $\sum_{k\ge
  n}\|R_k\|=\mc O(1/n^\gamma)$ for some $\gamma>1$ and $\Pi
R'(1)\Pi\neq 0$, then for all $n$
\[
T_n=\frac{1}{\mu}\Pi+\frac{1}{\mu^2}\sum_{k=n+1}^{+\infty}\Pi_k+E_n
\]
where $\mu$ is given by $\Pi R'(1)\Pi=\mu \Pi$, $\Pi_n=\sum_{l>n}\Pi
R_l\Pi$ and $E_n$ is a bounded operator satisfying
\[
\|E_n\|=\left\{\begin{array}{ll}
\mc O(1/n^{\gamma})& \mbox{if }\gamma>2\\
\mc O(\log n/n^2)& \mbox{if }\gamma=2\\
\mc O(1/n^{2\gamma-2})&\mbox{if }2>\gamma>1 
\end{array}\right.
\]
\end{theorem}

To apply the above result to our context, consider
$F_Y:=F^{\tau_Y}:Y\rightarrow Y$, and $P_Y:L^1(Y,\mb P_Y)\rightarrow
L^1(Y,\mb P_Y)$ its transfer operator.  Define operators $T_n,R_n:
L^1(Y,\mb P_Y)\rightarrow L^1(Y,\mb P_Y)$ for every $n\in\N_0$, in the
 following way
\[
T_0=\Id,\quad R_0=0,\quad T_n\phi=\chi_Y P_F^n(\chi_Y \phi), \quad
R_n\phi= P_F^n(\phi \chi_{\{\tau_Y=n\}})
\]
for $n\in\N$,  where $P_F: L^1(X, \mb P) \rightarrow L^1(X, \mb P)$ denotes the transfer operator for the skew product $F$.

\begin{corollary}\label{Cor:CorRenewal}
Assume there is a Banach space $(\mc B, \|\cdot\|)$, $\mc B\subset
L^1(Y,\mb P_Y)$, such that the operators $\{T_n\}_{n\ge 0}$ and
$\{R_n\}_{n\ge 1}$ satisfy the hypotheses of Theorem
\ref{Thm:RenGouez}. Then for any $h_1\in\mc B$ supported on $Y$, and
any
$h_2\in L^\infty(Y)$,
\[
\left|\int h_1h_2\circ F^nd \mb P_Y-\int h_1d \mb P_Y\int h_2 d\pi\right|\leq \mc O (n^{1-\gamma}).
\]
where $\pi$ is the stationary measure  for $F$ on $X$.
\end{corollary}

Let us say a few words about how this result follows from the previous theorem, given the definitions above.  We will see (in the proof of Proposition \ref{Prop:PropRenewal} below) that $P_Y = R(1)$,
$\mu =1/\pi [\frac 12, 1]$ and $d \pi_Y = \psi_0 d \mb P_Y$, with $ \Pi h  = ( \int h d \mb P_Y ) \psi_0 $ and the relation $\pi_Y \cdot \pi[\frac 12, 1] = \pi |_{[\frac 12, 1]}$.  We can then write 

\begin{align*} \int h_1h_2\circ F^nd  \mb P_Y &=2 \int \chi_Y  h_1  h_2 \circ F^n d \mb P \\
&= 2 \int P_F^n(\chi_Y  h_1)h_2  d \mb P\\
&= 2 \int \chi_Y P_F^n(\chi_Y h_1)h_2 d  \mb P\\
&= \int T_n(h_1) h_2 d  \mb P_Y \\
\end{align*}

We now apply the expansion of $T_n$ to obtain
\[
\int h_1h_2\circ F^nd  \mb P_Y =  \pi [\frac 12, 1] \left (\int h_1 d\mb P_Y\right ) \psi_0 h_2 d \mb P_Y + \textnormal{H.O.T}= \int h_1 d \mb P_Y  \int  h_2 d \pi  + H.O.T,
\]
where the higher order terms arise from the second and third terms in the expansion of $T_n$ and both decay with a rate upper bounded by $\mc O(n^{1-\gamma})$ for all three ranges of $\gamma$ in Theorem \ref{Thm:RenGouez}).
 
Corollary \ref{Cor:CorRenewal} now gives Theorem \ref{Thm:DecCorr} for the restricted case of $\psi$ and $\phi$ supported on $[\frac 12, 1]$ since $d \mb P_Y = 2 \mb P |_{[\frac 12,1]}$ in both integrals. We can extend to the case where $\phi$ is supported on $[0,1]$  as follows: Set $h_1( \omega, x) = \phi(2x-1)$ for $x \in [\frac 12, 1]$,  $h_1(\omega, x ) = 0$ for $x \in [0,\frac 12]$ and $h_2(\omega, x) = \psi(x)$.  Observe that $P_F h_1(\omega, x) = \frac 12 \phi(x)$ so $\int h_1\, d\mb P_Y = 2 \int \, h_1 d\mb P= \int \phi d \mb P$ and the correlation integral above becomes  
$$ \int h_1 h_2\circ F^{n}d  \mb P_Y=\int \phi \psi\circ F^{n-1}d  \mb P,  $$
leading to the result stated in Theorem \ref{Thm:DecCorr} provided $\psi$ is supported on $[\frac 12,1]$.   The final extension to fully supported $\psi$ can be established using the method detailed in Gou\"ezel \cite{gouezel2004sharp}, Theorem 6.9. 
    
The proposition below shows that the hypotheses of Theorem
\ref{Thm:RenGouez} and Corollary \ref{Cor:CorRenewal} are satisfied in our setup by the Banach space $\mc
B$ of functions on $Y$ that are a) constant w.r.t.\ $\omega$ and b) Lipschitz w.r.t.\ the spatial variable $x$. To simplify notation, we indicate these functions in terms of $x$ only and write

\[
\mc B:=\{\phi:[1/2,1]\rightarrow \mb C:\; |\phi|_{\Lip}<\infty\}
\] 
where 
\[
|\phi|_{\Lip}=\sup_{\substack{x,y\in[1/2,1]\\ x\neq y}}\frac{|\phi(x)-\phi(y)|}{|x-y|}
\]
is the Lipschitz semi-norm, and $\mc B$ is endowed with the norm
$\|\phi\|:=|\phi|_{\Lip}+|\phi|_\infty$.  Notice that if $\phi \in \mc B$  the function 
$P_F \phi$ is also constant w.r.t. $\omega$. In particular, we can compute the
skew product transfer operator as 
$$P_F \phi (x, \omega)= P_F \phi (x) = \int_{[\alpha,\beta]} P_\gamma \phi (x) d\nu(\gamma),$$
 where $P_\gamma$ denotes the transfer operator associated to $f_\gamma$  on $[0,1]$ with respect to $m$.
 We will see this leads to a useful simplification when computing higher powers $P_F^k$  and the induced transfer operator below. 

\begin{proposition} \label{Prop:PropRenewal}
Let $\mc B$ be as above. The maps $T_n$ and $R_n$ are bounded operators on $\mc B$.
Suppose that condition \eqref{Eq:Assum1} is satisfied for some $\gamma>1$.
The series $T(z)=\Id+\sum_{n\ge 1}z^n T_n$ converges on
$\mb D$, and:
\begin{itemize}
\item[(i)] $\Id-R(z)$ is invertible for every $z\in\bar{\mb D}\backslash \{1\}$;
\item[(ii)] $T(z)=(\Id-R(z))^{-1}$ for $z\in\mb D$, where $R(z)=\sum_{n\ge
  1}z^n R_n$ and $\sum_{n\ge 1}\|R_n\|<\infty$;
\item[(iii)] $R(1)$ has a spectral gap, i.e. there is $\Pi$ with
  $\Pi^2=\Pi$, $\dim\Im \Pi=1$, and there is $N$
  satisfying $\Pi N=N\Pi=0$, $\sigma(N)<1$, such that $R(1)=\Pi+N$;
  The non-degeneracy condition $\Pi R'(1)\Pi\neq 0$ holds.
\item[(iv)] $\sum_{k>n}\|R_k\|=\mc O(n^{-\gamma})$.
\end{itemize}
\end{proposition}

\begin{proof}[Proof of Proposition \ref{Prop:PropRenewal}]
Fix any $k\in\N$, a $n\ge k$, and any sequence
$(\omega_{0}\ldots \omega_{n-1})$ and write $f^n$ for $f^n_{(\omega_{0}\ldots \omega_{n-1})}$. 
Define $\mc
J^k_{(\omega_{0}\ldots \omega_{n-1})}$ to be the collection of
maximal subintervals of $[\frac 12,1]$ where $f^n$ is continuous, 
returning to $[\frac 12,1]$ for the 
$k$th time at time $n$ under $f^n$. If $J\in\mc
J^k_{(\omega_{0}\ldots \omega_{n-1})}$, the points of $J$ have the same
return times up to the $k$th return and $J$ is mapped injectively and
onto $[\frac 12,1]$ under $f^n$. Pick
$\psi\in \mc B$, then
\[
P^k_Y\psi=\int_{[\alpha,\beta]^{\N_0}}d\nu^{
  \N_0}(\omega_{0}\ldots \omega_{n-1}\ldots) \sum_{n\ge
  k}\sum_{J}P_{\omega_{n-1}}\ldots P_{\omega_{0}}\left(\psi\chi_{J}\right)(x)
\]
where the second sum is over $J\in \mc
J^k_{(\omega_{0}\ldots \omega_{n-1})}$ and again we indicate by $P_{\omega_i}$
the transfer operator of $f_{\omega_i}$.  We are going to prove that
$P_Y$ satisfies a Lasota-Yorke inequality.  First of all notice
that
\[
P_{\omega_{n-1}}\circ\ldots \circ P_{\omega_{0}}(\psi\chi_J)(x)=\psi\circ
f_{(\omega_{0}\ldots \omega_{n-1}),J}^{-n}(x)
\partial_xf_{(\omega_{0}\ldots\omega_{n-1}),J}^{-n}(x)
\]
with $f_{(\omega_{0}\ldots \omega_{n-1}),J}^{n}$ indicating the
restriction of the function to $J$. Given two points $x,y\in[1/2,1]$
\begin{align}
&\left|\psi\circ
f_{(\omega_{0}\ldots \omega_{n-1}),J}^{-n}(x)
\partial_xf_{(\omega_{0}\ldots \omega_{n-1}),J}^{-n}(x)-\psi\circ
f_{(\omega_{0}\ldots \omega_{n-1}),J}^{-n}(y)
\partial_xf_{(\omega_{0}\ldots \omega_{n-1}),J,}^{-n}(y)\right|\leq \label{Eq:A+BEst}
\\
&\quad\quad\leq\left|\partial_xf_{(\omega_{0}\ldots \omega_{n-1}),J}^{-n}(x)
\right|\left|\psi\circ
f_{(\omega_{0}\ldots \omega_{n-1}),J}^{-n}(x)-\psi\circ
f_{(\omega_{0}\ldots \omega_{n-1}),J}^{-n}(y)\right|+\nonumber\\ 
&\quad\quad\quad+\left|\psi\circ
f_{(\omega_{0}\ldots \omega_{n-1}),J}^{-n}(y)\right|\left|
\partial_xf_{(\omega_{0}\ldots \omega_{n-1}),J}^{-n}(x)-
\partial_xf_{(\omega_{0}\ldots \omega_{n-1}),J}^{-n}(y)\right|=\nonumber\\
&\quad\quad=:A+B.\nonumber
\end{align}
To bound $A$ we show first that there is $M>0$ uniform in the choice
of $J$ such that
\[
|\partial_xf_{(\omega_{0}\ldots \omega_{n-1}),J}^{-n}(x)|\leq M |J|. 
\]
In fact, recall that $J$ is mapped bijectively onto $[\frac 12,1]$. By the
mean value theorem, there is $\xi\in [\frac 12,1]$ such that
\[
2|J|=\partial_xf_{(\omega_{0}\ldots \omega_{n-1}),J}^{-n}(\xi).
\] 
Combining this with the bound on distortion from Corollary
\ref{Cor:Dist}, there is a constant $K$ such that
\begin{align*}
|\partial_xf_{(\omega_{0}\ldots \omega_{n-1}),J}^{-n}(x)| 
&=
\left|\frac{\partial_xf_{(\omega_{0}\ldots \omega_{n-1}),J}^{-n}(x)}
{\partial_xf_{(\omega_{0}\ldots \omega_{n-1}),J}^{-n}(\xi))}\right|
\left|\partial_xf_{(\omega_{0}\ldots \omega_{n-1}),J}^{-n}(\xi)\right|\\
&=
\left|\frac{\partial_xf_{(\omega_{0}\ldots \omega_{n-1}),J}^{n}
(f_{(\omega_{0}\ldots \omega_{n-1}),J}^{-n}(\xi))}{\partial_x
f_{(\omega_{0}\ldots \omega_{n-1}),J}^{n}(f_{(\omega_{0}\ldots \omega_{n-1}),J}^{-n}(x))}
\right||\partial_xf_{(\omega_{0}\ldots \omega_{n-1}),J}^{-n}(\xi)|\\ 
&\le
K |J|
\end{align*}
This and the fact that
\[
|\psi\circ f^{-n}_{(\omega_{0}\ldots \omega_{n-1}),J}(x)-\psi\circ 
f^{-n}_{(\omega_{0}\ldots \omega_{n-1}),J}(y)|\leq 2^{-k}|\psi|_{\Lip}|x-y|
\]
imply that
\[
A\leq M|J| 2^{-k}  |\psi|_{\Lip}|x-y|.
\]
To bound $B$ first of all notice that
\begin{align*}
\left|\psi\circ f_{(\omega_{0}\ldots \omega_{n-1}),J}^{-n}(y)\right|
&\leq |J|^{-1}\int_{J}|\psi|+2^{-k}|\psi|_{\Lip}
\end{align*}
where we used that the diameter of $J$  is less then $2^{-k}$ given that its points 
have $k$ common return times with respect to sequences from the cylinder 
$[\omega_{0}\ldots \omega_{n-1}]$. Then notice that
\begin{align*}
&\left| \partial_xf_{(\omega_{0}\ldots \omega_{n-1}),J}^{-n}(x)-\partial_x
f_{(\omega_{0}\ldots \omega_{n-1}),J}^{-n}(y)\right|\\
&=| \partial_x
f_{(\omega_{0}\ldots \omega_{n-1}),J}^{-n}(x)|\left| 1-\frac{\partial_x
f_{(\omega_{0}\ldots \omega_{n-1}),J}^{-n}(y)}{ \partial_x
f_{(\omega_{0}\ldots \omega_{n-1}),J}^{-n}(x)}\right|\\
&=| \partial_xf_{(\omega_{0}\ldots \omega_{n-1}),J}^{-n}(x)|\left| 1-\frac{\partial_x
f_{(\omega_{0}\ldots \omega_{n-1}),J}^n(f_{(\omega_{0}\ldots \omega_{n-1}),J}^{-n}(x))}
{\partial_xf_{(\omega_{0}\ldots \omega_{n-1}),J}^n
(f_{(\omega_{0}\ldots \omega_{n-1}),J}^{-n}(y))}\right|\\
&\le M'|J||x-y|,
\end{align*}
where we made use of Proposition \ref{Prop:UnifBndDist} in the last line.
We conclude that 
\[
B\le M'\left(\int_{J}|\psi|dm(x)+|J|2^{-k}|\psi|_{\Lip}\right)|x-y|
\]
Putting the bounds for $A$ and $B$ together, we conclude that there is 
$M''>0$ (uniform) such that for every $k\in\N$, $n\ge k$, 
$(\omega_{0}\ldots \omega_{n-1})$, and $J\in\mc J^k_{(\omega_{0}\ldots \omega_{n-1})}$
\begin{equation}\label{Eq:LipEstPiece}
|P_{\omega_{n-1}}\circ\ldots \circ P_{\omega_{0}}(\psi\chi_J)|_{\Lip}
\leq M''\left(\int_{J}|\psi|dm(x)+|J|2^{-k}|\psi|_{\Lip}\right).
\end{equation}

Applying this to $\mathcal J^1_{\omega_0\ldots\omega_{n-1}}$, which consists of a single interval,
$J^1_{\omega_0,\ldots,\omega_{n-1}}$ say, we see
$$
|P_{\omega_{n-1}}\circ\ldots\circ P_{\omega_0}(\cdot \chi_J)|_\text{Lip}\le M''|J^1_{\omega_0\ldots
\omega_{n-1}}|.
$$
We note
$$
R_n(\cdot)=\int_{[\alpha,\beta]^n}P_{\omega_{n-1}}\circ\ldots\circ 
P_{\omega_0}(\cdot \chi_{J^1_{\omega_0
\ldots\omega_{n-1}}})
\,d\nu^n,
$$
so that 
$$
\|R_n\|\le M''\int_{[\alpha,\beta]^n}|J^1_{\omega_0\ldots\omega_{n-1}}|\,d\nu^n.
$$
Since for each $(\omega_0,\ldots)\in [\alpha,\beta]^{\N_0}$, the sets 
$J^1_{\omega_0\ldots\omega_{n-1}}$ form a countable partition of $[\frac 12,1]$, 
it follows that $\sum\|R_n\|<\infty$ as required. 

By the same calculation,
\begin{align*}
\sum_{k>n}\|R_k\|&\le M''\int_{[\alpha,\beta]^{\N_0}}
\sum_{k>n}|J^1_{\omega_0\ldots\omega_{k-1}}|
\,d\nu^{\N_0}(\bo \omega)\\
&=M''\, \mb P_Y(\{(\bo\omega,x)\in Y\colon \tau_Y(\bo\omega,x)>n\})=
\mc O(n^{-\gamma})
\end{align*}
by condition \eqref{Eq:Assum1}, establishing \textbf{point (iv)}.

For every $\bo \omega=(\omega_0\omega_1\ldots )$, let 
$\mc J^k_{\bo\omega}$ be $\bigcup_{n\ge k}\mc J^k_{(\omega_0\ldots \omega_{n-1})}$, 
the countable partition of $[\frac 12,1]$ according to the first $k$ returns to $[\frac 12,1]$ under 
the maps $f_\omega^n$. One obtains
\begin{equation}\label{eq:essradbound}
\begin{split}
|P^k_Y\psi|_{\Lip}&\leq \int_{[\alpha,\beta]^{\N_0}}d\nu^{ \N_0}
( \omega_{0}\ldots \omega_{n-1}\ldots) \sum_{n\ge k}
\sum_{J}|P_{\omega_{n-1}}\ldots P_{\omega_{0}}\left(\psi\chi_{J}\right)|_{\Lip}\\
&\leq M'' \sum_{n\ge k}\int_{[\alpha,\beta]^{n}}
d\nu^{ n}(\omega_{0}\ldots \omega_{n-1})
\sum_{J\in\mc J^k_{(\omega_{0}\ldots \omega_{n-1})}}
\left(2^{-k}|J||\psi|_{\Lip}+\int_J|\psi|\right)\\
&=M''\int_{[\alpha,\beta]^{\N_0}}d\nu^{ \N_0}(\bo\omega)
\sum_{J\in\mc J^k_{\bo\omega}}\left(2^{-k}|J||\psi|_{\Lip}+\int_J|\psi|\right)\\
&=M'' 2^{-k}|\psi|_{\Lip}\int_{[\alpha,\beta]^{\N_0}}d\nu^{ \N_0}(\bo\omega)
\sum_{J\in\mc J^k_{\bo\omega}}|J|+M''\int_{[\alpha,\beta]^{\N_0}}
d\nu^{ \N_0}(\bo\omega)\sum_{J\in\mc J^k_{\bo\omega}}\int_J|\psi|\\
&=M'' 2^{-k}|\psi|_{\Lip}\,\frac{1}{2}+M''\int_{[1/2,1]}|\psi|,
\end{split}
\end{equation}
where in the first line, the sum is over those intervals whose $k$th return occurs at time $n$.  An elementary calculation shows 
$$\frac{1}{1-1/2} \int_{1/2}^1 P^k_Y \psi \, dm - |P^k_Y\psi|_{\Lip}   \leq P^k_Y\psi(x) \leq \frac{1}{1-1/2} \int_{1/2}^1  P^k_Y\psi \, dm+ |P^k_Y\psi|_{\Lip},$$
for every $x \in [1/2, 1]$ which, when combined with $| \int_{1/2}^1 P^k_Y\psi |  = 
|\int_{1/2}^1 \psi | \leq 1/2 |\psi|_\infty$, yields $| P^k_Y\psi|_\infty \leq |\psi|_\infty + |P^k_Y\psi|_{\Lip} $.  Combining this with the estimate above 
gives the following Lasota-Yorke inequality. 
\begin{equation}\label{eq:LY}
\| P^k_Y \psi \| \leq  \frac{M''}{2^k} \| \psi\| + (M'' +1) |\psi |_\infty.
\end{equation}

The operator $P_Y$ is 
therefore quasi-compact (by Hennion's theorem
\cite{hennion1993theoreme}).

We now show that $P_Y$ has unique fixed point, giving rise to an
invariant mixing measure $\nu^{\N_0}\otimes \mu$ for $F_Y$ where $\mu$
is absolutely continuous with respect to Lebesgue. In particular,
consider the cone of positive functions $\mc
C_a=\{\psi:[1/2,1]\rightarrow \R^+>0:\;\psi(x)/\psi(y)\leq
e^{a|x-y|}\}$. Since $\psi\in \mc C_a$ is independent of $\omega$
\[
P_Y\psi(x)=\int_{[\alpha,\beta]^{\N_0}} \sum_{n\ge
  1}P_{\omega_{n-1}}\ldots P_{\omega_{0}}\left(\psi\chi_{J_{(\omega_{0}\ldots
    \omega_{n-1})}}\right)(x) d\nu^{\N_0}(\omega_{0}\ldots \omega_{n-1}\ldots).
\]
Since $f^n_{(\omega_{0}\ldots \omega_{n-1})}$ has bounded distortion
and expansion both uniform in $n$ and $(\omega_{0}\ldots
\omega_{n-1})$, there is $a>0$ sufficiently large such that if
$\psi\in\mc C_a$
\[
\frac{P_{\omega_{n-1}}\ldots
  P_{\omega_{0}}\left(\psi\chi_{J_{(\omega_{0}\ldots
      \omega_{n-1})}}\right)(x)}{P_{\omega_{n-1}}\ldots
  P_{\omega_{0}}\left(\psi\chi_{J_{(\omega_{0}\ldots
      \omega_{n-1})}}\right)(y)}\leq e^{a'|x-y|}\;\;\;\; \forall
x,y\in[1/2,1],\;\forall n,\; \forall (\omega_{0}\ldots \omega_{n-1})
\]
with $a'<a$. It is easy to conclude that if $\psi\in\mc C_a$, then
$P_Y\psi\in\mc C_{a'}$. It is standard to conclude that $P_Y$ fixes a
direction in $\mc C_a$, giving a mixing invariant absolutely continuous probability measure $d\pi_Y = \psi_0 d \mb P_Y$ (see \cite{liverani1995decay} pp 244-250).  It is also easy to see that $\psi\in \mc B$. This fact
together with quasi-compactness of $P_Y$ and mixing implies that the
operator has a spectral gap.

{\bf Point (i)} We first deal with $|z|<1$. Pick $z\in\mb D$. Notice that
\begin{align*}
R(z)^k\phi(x)&=\sum_{n\ge 1}\int z^n\sum_J P_{\omega_{n-1}}\circ\ldots
\circ P_{\omega_0}(\phi\chi_J)\,d\nu^\N(\bo\omega),
\end{align*}
where the summation is over $\mathcal J^k_{\omega_0,\ldots,\omega_{n-1}}$,
the collection of intervals (depending on $\bo\omega$) returning to $Y$ for the $k$th
time at the $n$th step as above.
%that $R(z)\phi=P_Y(z^{\tau_Y}\phi)$. In fact,
%\[
%R(z)\phi=\sum_{n\ge 1} z^nP_Y(\phi\chi_n)=\sum_{n\ge 1}
%P_Y(z^{\tau_Y}\phi\chi_n)=P_Y(z^{\tau_Y}\phi).
%\]
%Note that while $z^{\tau_Y}\phi$ is not a Lipschitz function,
%its image is Lipschitz (by the same proof as above). 
%Similarly $(R(z))^k\phi=P_Y^k(z^{\tau_Y^k}\phi)$, where $\tau_Y^k$ is the
%time of the $k$th return to $Y$. 
The terms corresponding to $A$ and $B$
in \eqref{Eq:A+BEst} are scaled by $|z|^n$, where $n\ge k$ is the time of the $k$th
return. In particular, arguing as in \eqref{eq:essradbound}, we see
$\|R(z)^k\|\le M|z|^k$, where $M$ does not depend on $k$ or $z$. Hence for
$|z|<1$, $\sum_{n\ge 0}R(z)^n$ is convergent, and $\Id-R(z)$ is invertible.

For $|z|=1$, one can show that $R(z)$ satisfies
a Lasota-Yorke inequality as we did for $R(1)=P_Y$.
This implies that the essential spectrum of $R(z)$ is contained in the
open disk $\mb D$, and $\Id-R(z)$ is invertible if and only if $1$ is
not an eigenvalue of $R(z)$. Now let $z$ lie on the unit circle but $z\ne 1$. 
Suppose for a contradiction that $\phi$ is a Lipschitz function that is an eigenfunction of
$R(z)$ with eigenvalue 1. Then taking absolute values and using the triangle 
inequality, we see $R(1)|\phi|\ge |\phi|$, this inequality being strict
unless for each $x$, almost every term contributing to $R(z)\phi$ has the same
argument. Since $\int R(1)|\phi|(x)\,dm(x)=\int |\phi(x)|\,dm(x)$, this implies
$R(1)|\phi|=|\phi|$, so that $|\phi|$ is the leading eigenvector, $g$ say, of $P_Y$. 
We write $\phi(x)=g(x)h(x)$ where $|h(x)|=1$ for all $x$. Since $g\in\mathcal C_a$, 
it is bounded below, so that $h$ must be Lipschitz. The condition for equality in the 
triangle inequality implies that $h(F_Y(\bo\omega,x))=h(x)z^{\tau_Y(\bo\omega,x)}$ for
$\PP_Y$-a.e.\ $(\bo\omega,x)$. In particular for $\nu^\N$-a.e.\ $\bo\omega$,
for $m$-a.e. $x$ in each $J^1_{\omega_0,\ldots\omega_{n-1}}$, we have 
$$
h(F_Y(\bo\omega,x))=h(x)z^n.
$$
Since both sides of the equality are Lipschitz functions, this holds for \emph{all}
$x\in J^1_{\omega_0,\ldots\omega_{n-1}}$. Since the $J^1_{\omega_0,\ldots\omega_{n-1}}$ 
are arbitrarily short for large $n$, we see that $\var_J(h)$ becomes arbitrarily small. But the 
above equality and the fact that $F_Y(\bo\omega,\cdot)$ maps $J$ onto $[\frac 12,1]$
imply that $\var_J(h)=\var_{[\frac 12,1]}(h)$. Hence $h$ is a constant and
we may assume $\phi=g$. Now $R(z)g=R(1)g=g$ implies that $z=1$.

{\bf Point (ii)} One can verify the renewal equations
\begin{align*}
T_n&=T_0R_n+T_1R_{n-1}+\ldots+T_{n-1}R_1\\
T_n&=R_nT_0+R_{n-1}T_1+\ldots+R_1 T_{n-1}.
\end{align*}
Now since $\sum_{n\ge 1}R(z)^n$ converges, we deduce that $(\Id-R(z))$ is invertible.
The coefficient multiplying $z^n$ in this last
sum is 
$$
\sum_{k=1}^n\sum_{i_1+\ldots+i_k=n}R_{i_1}R_{i_2}\ldots R_{i_k}=T_n,
$$
so $T(z)=(\Id-R(z))^{-1}$.
This also establishes the boundedness of $T_n$ for each $n$. 

{\bf Point (iii)} Notice that $R(1)=\sum_{n=1}^{\infty}R_n=P_Y$, so that there is the required
spectral gap and simple eigenvector with eigenvalue 1 as shown above. We have
$R'(1)=\sum_{n=1}^\infty nR_n=\sum_{n=1}^\infty\sum_{m\ge n}R_m$,
so that this is a bounded operator by point (iv). 
Note also that $R(1)$ preserves integrals, so that $\Pi$ also preserves integrals
and the operators $R_n$ preserve the class of non-negative functions. 
Since $R'(1)=R(1)+\sum_{n>1}(n-1)R_n$, It follows that $\Pi R'(1)\Pi$
is non-zero as required. 
%{\bf Point (iv)} Follows from the estimates on the return sets similar to
%condition \eqref{Eq:Assum}. \SAY{Proof?!! Prove last part of (i) at the same time...}
\end{proof}

\subsection{Central Limit Theorem and Convergence to Stable Laws}\label{SS:CLTstab}\label{Sec:ConStabLaws}

The fact that $P_Y$ has a spectral gap allows us to use the
Nagaev-Guivarc'h method (as outlined in \cite{aaronson2001local}
\cite{Gouezel2}) to show that suitably rescaled Birkhoff sums
converge to stable laws.  In this approach, one first proves that suitably
rescaled Birkhoff sums of observables in the induced system ($F_Y$) converge
to stable laws, and then use this result to show that the same limit
theorem holds for the unfolded system ($F$). We are not going to give full details 
of the proofs, but we are going to present some computations needed in the
particular case that we are treating, and direct the reader to the relevant literature for the 
remaining standard part of the argument.

Let us call \[S^n_Y\psi:=\psi+\psi\circ F_Y+\ldots +\psi\circ F_Y^{n-1}\] the
Birkhoff sums with respect to $F_Y$.  We assume, as in the statement of 
Theorem \ref{Thm:LimitUnfold} that $\phi$ is a Lipschitz function on $[0,1]$ satisfying $\phi(0)>0$. 
Recall that given $\phi$,
$\phi_Y(\bo \omega,x)=\sum_{i=0}^{\tau_Y(\bo \omega,x)-1}\phi\circ
F^i(\bo \omega,x)$, and $G$ is the cumulative distribution function of $\phi_Y$
with respect to $\mathbb P_Y$: that is $G(t)=\mathbb P_Y(
\{(\bo\omega,x)\in\Omega\times[\frac 12,1]\colon 
\phi_Y(\bo\omega,x)\le t\}$.

\begin{proposition}\label{prop:inducedstablelaw}
Assume $G$ is such that $L(t):=t^p(1-G(t))$ with $p\neq 2$ is slowly varying as $t\to+\infty$
and $G(t)=0$ for large negative $t$. Then there is a stable law $\mc Z$ a sequence $A_n$, and a sequence $B_n$ satisfying
$nL(B_n)=B_n^p$ such that
\[
\lim_{n\rightarrow\infty}\frac{S^n_Y\phi_Y-A_n}{B_n}\rightarrow\mc Z
\]
where the convergence is in distribution. If $p>2$, then $\mc Z$ is a Gaussian, while if $p<2$, $\mc Z$ is a stable law of index $p$. 
\end{proposition}
\begin{proof}
Calling $P_{Y,t}(\cdot):=P_Y(e^{it\phi_Y}\cdot)$, it follows that
\[
\mb E[e^{itS^n_Y\phi_Y}]=\int P^n_{Y,t}1d\mb P.
\]
The relation above provides the foundation to the Nagaev-Guivarc'h approach 
that  recovers information on the characteristic function of $S^n_Y$
from spectral properties of the family $(P_{Y,t})_t$.

\paragraph{Step 1} The operator $P_{Y,t}$ has a spectral gap in $\mc B$ for all
sufficiently small $t$.  We show how to prove this fact directly with  computations analogous
to those carried out in the proof of point (iii) in Proposition
\ref{Prop:PropRenewal}. However, in that proposition $\phi_Y=\tau_Y$,
i.e. $\phi=1$, and therefore $\phi_Y$ is piecewise constant on every fibre $\Omega\times\{x\}$. 
When this is not the case, $\phi_Y$ remains piecewise Lipschitz on fibres, where the fibres are
partitioned according to the return time to $Y$. However, the Lipschitz constants on these
pieces are typically not bounded, so the previous argument needs some extra care. 

One needs to estimate $|P^k_{Y,t}(\psi)|_{\Lip}$, given by
$$
P^k_{Y,t}\psi(x)=\int d\nu^{\N_0}(\bo\omega)
\sum_{n\ge k}\sum_J P_{\omega_{n-1}}\cdots P_{\omega_0}(\chi_Je^{itS_k\phi_Y}\psi)(x),
$$
where the $J$ summation is over $\mathcal J^k_{\omega_0,\ldots,\omega_{n-1}}$, the 
collection of intervals returning to $Y$ for the $k$th time at time $n$.
Now fix $x<y$ in $[\frac12,1]$; $\bo\omega\in[\alpha,\beta]^{\N_0}$, $k\le n$
and $J\in\mathcal J^k_{\omega_0,\ldots,\omega_{n-1}}$. Let $x',y'$
be respectively the points in $(f_{\bo\omega}^n)^{-1}(x)$ and $(f_{\bo\omega}^n)^{-1}(y)$ 
that lie in $J$. Then the contribution to $P^k_{Y,t}\psi(y)-P^k_{Y,t}\psi(x)$ coming from the 
interval $J$ is
$$
\frac{e^{itS_k\phi_Y(y')}\psi(y')}{(f_{\bo\omega}^n)'(y')}-
\frac{e^{itS_k\phi_Y(x')}\psi(x')}{(f_{\bo\omega}^n)'(x')}.
$$
We estimate the absolute value of this quantity by
$$
\frac{t|S_k\phi_Y(y')-S_k\phi_Y(x')|\,\|\psi\|_\Lip}
{(f_{\bo\omega}^n)'(y')}
+\frac{|\psi(y')-\psi(x')|}{(f_{\bo\omega}^n)'(y')}+|\psi(x')|\left|\frac{1}{(f_{\bo\omega}^n)'(y')}
-\frac{1}{(f_{\bo\omega}^n)'(x')}\right|
$$
The combined contribution from the second and third terms as $n$ runs over $k,k+1,\ldots$
and $J$ runs over $\mathcal J^k_{\bo\omega}$ and integrated over $\bo\omega$
is estimated exactly as in Proposition
\ref{Prop:PropRenewal}. It remains to estimate the combined contribution from the first term.
Since the $f_\omega$'s are non-contracting and the $k$th return time is $n$,
the first term $|S_k\phi_Y(y')-S_k\phi_Y(x')|$
may be estimated by $n|\phi|_\Lip||x-y|$.   As before, the distortion estimates give
$1/(f_{\bo\omega}^n)'(y)\le K|J|$.
Hence the combined contribution to $|P^k_{Y,t}(\psi)(y)-P^k_{Y,t}(\psi)(x)|$ 
coming from the first terms in the above display (as $n$, $J$ and $\bo\omega$ vary) is bounded above by 
\begin{align*}
&Kt\|\phi\|_\Lip\|\psi\|_\Lip|x-y|\int d\nu^{\N_0}(\bo\omega)
\sum_n\sum_{J\in\mathcal J^k_{\omega_0,\ldots,\omega_{n-1}}} n|J|\\
=&Kt\|\phi\|_\Lip\|\psi\|_\Lip|x-y| \int \tau_k(\bo\omega,x)\,d\mathbb P(\bo\omega,x),
\end{align*}
where $\tau_k(\bo\omega,x)$ denotes the $k-$th return time to $Y$. 
The measure $\mathbb P$ agrees  with $\nu^{\N_0}\otimes\pi_Y$ (where $\pi_Y$ is the absolutely
continuous invariant measure for $F_Y$, i.e. the density of $\pi_Y$ is the unique fixed point of $P_Y$ on $\mc B$)
 up to a multiplicative factor that is uniformly 
bounded above and below.
Hence the above displayed quantity may be estimated by
\begin{align*}
&\leq K't\|\phi\|_\Lip\|\psi\|_\Lip|x-y|\int\tau_k(\bo\omega,x)\,d\nu^{\N_0}\otimes\pi_Y\\
&=kK't\|\phi\|_\Lip\|\psi\|_\Lip|x-y| \int \tau_1(\bo\omega,x)\,d\nu^{\N_0}\otimes\pi_Y
\end{align*}
By Kac's Lemma, the integral is finite, so the above reduces to 
$Ct\|\psi\|_\Lip|x-y|$. Taking $t$ sufficiently small, we see that the additional contribution to
the estimate of $|P^k_{Y,t}(\psi)|_\Lip$ from that appearing in the earlier proposition is
$Ct\|\psi\|_\Lip$. In particular, for sufficiently small $t$, one obtains a Lasota-Yorke inequality analogous to the one in (\ref{eq:LY}).

\paragraph{Step 2}The family of operators 
$(P_{Y,t})$ (acting on $\mathcal B$) is continuous in $t$. To see this, we estimate
$\|(P_{Y,t}-P_{Y,s})(\psi)\|_\Lip$. As above, $|(P_{Y,t}-P_{Y,s})(\psi)(y)- (P_{Y,t}-P_{Y,s})(\psi)(x)|$
is expressed as an integral over $\nu^{\N_0}$ of a countable sum (one term for each $n$). 
Fix $x<y$ in $[\frac 12,1]$, $\bo\omega$ and $n$ as above. Since we are considering first
returns, there is exactly one interval, $J$, in $\mathcal J^1_{\omega_0,\ldots,\omega_{n-1}}$. 
Letting $x'$ and $y'$ be the preimages of $x$ and $y$ under $(f_{\bo\omega}^n)^{-1}$ in $J$,
the corresponding contribution to $\Delta:=|(P_{Y,t}-P_{Y,s})(\psi)(y)- (P_{Y,t}-P_{Y,s})(\psi)(x)|$ is
estimated by
\[
\left|\frac{e^{is\phi_Y(y')}(e^{i(t-s)\phi_Y(y')}-1)\psi(y')}{(f_{\bo\omega}^n)'(y')}-
\frac{e^{is\phi_Y(x')}(e^{i(t-s)\phi_Y(x')}-1)\psi(x')}{(f_{\bo\omega}^n)'(x')}\right|
\]
(where we dropped the $\bo \omega$-dependence of $\phi_Y$ for clarity of the notation), which we bound as a difference of products using the triangle inequality as usual, giving rise to the sum
of four terms.
We use the estimates $|(e^{i(t-s)\phi_Y(y')}-e^{i(t-s)\phi_Y(x')})/(f_{\bo\omega}^n)'(\xi)|\le n|t-s| K\|\phi\|_\Lip|y-x|$ for any $\xi$ between $x'$ and $y'$;
$|e^{i(t-s)\phi_Y(x')}-1|\le n|t-s|K\|\phi\|_\Lip$; $|1/(f_{\bo \omega}^n)'(y')-1/(f_{\bo\omega}^n)(x')|
\le K|J||x-y|$ (obtained from the distortion bound).
Combining these estimates, we find that the contribution to $\Delta$ coming from the $n$th interval
$J_{\omega_0,\ldots,\omega_{n-1}}$ from each of the four terms described above is bounded above by
$Cn|t-s|\|\phi\|_\Lip|x-y||J|\|\psi\|_\Lip$. Summing over $n$ and integrating, we find
\begin{align*}
\|P_{Y,t}-P_{Y,s}\|_\Lip&\le 
C|t-s|\|\phi\|_\Lip\int_\Omega \sum_n n|J_{\omega_0,\ldots,\omega_{n-1}}|\,d\nu^{\N_0}(\bo\omega)\\
&=C|t-s|\|\phi\|_\Lip\int_{\Omega\times[\frac 12,1]}\tau_1(
\bo\omega,x)\,dm(x)\,d\nu^{\N_0}(\bo\omega).
\end{align*}
As in Step 1, the integral is finite using Kac's Lemma, so that $\|P_{Y,t}-P_{Y,s}\|_\Lip\le C'|t-s|$. 

\paragraph{Step 3} The spectral gap of $P_{Y,t}$ together with continuity 
of the family $\{P_{Y,t}\}_t$ implies that $P_{Y,t}$ has a simple eigenvalue  
$\lambda(t)$ such that $\lambda(t)\rightarrow 1$ as $t\rightarrow 0$. 

 If $p>2$, then the distribution random variable with distribution $G$ is square summable, and results from \cite{liverani1996central} show that the limit law is a Gaussian (see also \cite{Gouezel2}). 

If $p<2$, an expansion of $\lambda(t)$ in $t$ can be obtained following the proof of Theorem 5.1 
from \cite{aaronson2001local}. The expansion will depend on the distribution $G$. One of the main requirements here is that $G$ is in the domain of attraction of a stable law or, equivalently, that there are $L(x)$, a slowly varying function as $x\rightarrow +\infty$, and $c_1,c_2\ge 0$ with $c_1+c_2>0$ such that $G(x)=1-x^{-p}L(x)(c_1+o(1))$ as $x\rightarrow +\infty$ and $G(x)=|x|^{-p}L(-x)(c_2+o(1))$ for $x\rightarrow -\infty$.  These conditions are ensured by the assumptions on $G$ with $c_2=0$. One can then conclude  the statement of the proposition as outlined in 
Theorem 6.1 from \cite{aaronson2001local}.  The gist of the argument is  that 
$\int P_{Y,\frac{1}{B_n}}^n1$ approaches $\lambda(\frac{1}{B_n})^n$. This 
information and the expansion of $\lambda(t)$ found at the previous step can 
be used to determine the limit for the law of $\frac{S^n_Y\phi_Y-A_n}{B_n}$ 
when $n\rightarrow \infty$. 
 \end{proof}

The following general theorem allows us to \say{unfold} the limit theorem from the 
induced system $F_Y$ to the original system $F$.
\begin{theorem}[Theorem 4.8 \cite{Gouezel2}]\label{Thm:GouezUnfold}
Let $T:X\rightarrow X$ be an ergodic probability preserving map
w.r.t. the measure $\mu$, let $(M_n)$ and $(B_n)$ be two sequences
of integers which are regularly varying with positive indices, let
$A_n\in\R$, and let $Y\subset X$ be a subset of positive $\mu$
measure. Denote by $\mu_Y$ the probability measure induced on $Y$.
Let $\tau_Y$ be the first return time to $Y$, $T_Y:= T^{\tau_Y}$ the
induced transformation, $\phi:X\rightarrow \R$ a measurable function,
$\phi_Y(x):=\sum_{i=0}^{\tau_Y(x)-1}\phi\circ T^i(x)$, $S_Y^n
\phi_Y=\sum_{i=0}^{n-1}\phi_Y\circ T_Y^i(x)$.
Assume that $(S_Y^n\phi_Y-A_n)/B_n\rightarrow \mc Z$ in distribution,
that $M_n$ is such that $(S_Y^n\tau_Y-  n/ \mu(Y))/M_n$ is tight
and $\max_{0\leq k\leq M_n}|S^k_Y \phi_Y|/B_n$ tends in
probability to zero.
Then
$(S^n\phi-A_{\lfloor{n\mu(Y)}\rfloor})/B_{\lfloor{n\mu(Y)}\rfloor}\rightarrow
\mc Z$ in distribution.
\end{theorem}

Before proving Theorem \ref{Thm:LimitUnfold}, we show that the tails of $\phi_Y$ and $\tau_Y$
satisfy the hypotheses of Proposition \ref{prop:inducedstablelaw}.
Recall that $\nu$ is a probability measure supported on $[\alpha,\beta]$ with
$\alpha<1$; $\alpha$ in the topological support of $\nu$;
and $\phi$ is a Lipschitz function on $[0,1]$. For the rest of this argument we will assume $\phi(0)>0$. Then there is a bounded number of times one can have
$\phi(x_n) < 0$ which implies the condition $G(t)=0$ for all $t < t_0$. The case $\phi(0) <0$ will be addressed at the very end of this section. The following proposition shows that the the distribution $G$ of $\phi_Y$ is a regularly varying function independently of the choice of $\nu$, i.e. it can always be written as the product of a monomial and a slowly varying function. Furthermore, the index of the distribution depends only on the minimum of the topological support of $\nu$.
\begin{proposition}\label{prop:phiYinbasin}
Let $\nu$ and $\phi$ be as above. Then $G(t):=\mb P_Y\{(\bo\omega,x):\phi_Y(\bo\omega,x)\le t\}$
satisfies $1-G(t)=1/t^{1/\alpha}L(t)$, where $L(t)$ is a slowly varying function.
\end{proposition}
The proof of this proposition requires a couple of lemmas. Let 
\[
q(x)=\int_x^{\frac 12}\frac {\phi(t)}{t\left[\int (2t)^\gamma\,d\nu(\gamma)\right]}\,dt.
\]
In the function above, $t\left[\int (2t)^\gamma\,d\nu(\gamma)\right]$ is the average  size (integrating $\gamma$) of a one-step displacement $f_\gamma(t)-t$. The random i.i.d. nature of the composed maps implies that the number of steps needed to escape from the point $x\approx 0$ to the inducing set, concentrates around a given value, and $q(x)$ is expected to be a proxy for the value around which $\phi_Y(\bo \omega,x')$ concentrates, 
where $x'$ is the unique preimage of $x$ in $Y$.  As a convenient abuse of notation in the following computations, we will write, $\phi_Y(\bo \omega,x)$ when $x\approx 0$, with the convention that this means $\phi_Y(\bo \omega,x')$.
The following lemmas make the above heuristic precise.

\begin{lemma}\label{lem:rettimeest}
For all $\delta>0$, there exists an $x_0$ such that if $x\le x_0$,
then 
\begin{align*}
\nu^\N\{\bo\omega:\phi_Y(\bo\omega,z)\ge q(x)\}\ge 1-\delta&\quad\text{for all $z\le (1-\delta)x$}\\
\nu^\N\{\bo\omega: \phi_Y(\bo\omega,z)\ge q(x)\}\le \delta x&\quad\text{for all $z\ge (1+\delta)x$}.
\end{align*}

In particular, if $x\le x_0$ and $t=q(x)$, then 
$$
(1-2\delta)x\le \mathbb P_Y(\{(\bo\omega,y)\colon \phi_Y(\bo\omega,y)> t\})\le (1+2\delta)x.
$$
\end{lemma}

To prove this lemma we first show a concentration result for the number of steps needed to escape from the interval $[e^{-\sqrt{n}},e^{-\sqrt{n-1}}]$.
\begin{lemma}\label{lem:Intime}
Let the probability measure $\nu$ on $[\alpha,\beta]$ be as above. 
Let $I_n$ denote the interval $[e^{-\sqrt{n}},e^{-\sqrt{n-1}})$ and fix $\eta>0$.
For $x\in [e^{-\sqrt n},f_\alpha(e^{-\sqrt n}))$, let $T_n(\bo\omega,x)=\min\{k\colon
f_{\omega_{k-1}}\circ\ldots\circ f_{\omega_0}(x)>e^{-\sqrt{n-1}}\}$, that is the number of steps
spent in $I_n$.

For all sufficiently large $n$ and all $x\in [e^{-\sqrt n},f_\alpha(e^{-\sqrt n}))$,
$$
\nu^\N\{\bo\omega:T_n(\bo\omega,x)\in [(1-\eta)N_n,(1+\eta)N_n]\}\ge 1-2^{-n},
$$
where $N_n=(e^{-\sqrt {n-1}}-e^{-\sqrt n})/[ e^{-\sqrt n} \int (2e^{-\sqrt n})^\gamma\,d\nu(\gamma)]$.
\end{lemma}

Notice that when a random orbit enters $I_n$ from $I_{n+1}$ the first point of $I_n$ that is encountered
lies in $[e^{-\sqrt n},f_\alpha(e^{-\sqrt n}))$. Hence the lemma is estimating the number of steps spent
in $I_n$ when entering from the left. 

\begin{proof}
Let $a_n=e^{-\sqrt n}$. If $x\in I_n$, the size of each step is $x(2x)^\gamma$ where 
$\gamma\in[\alpha,\beta]$. We approximate this above and below by $a_{n-1}(2a_{n-1})^\gamma$
and $a_n(2a_n)^\gamma$. The ratio of these step sizes is bounded above by $(a_{n-1}/a_n)^{1+\beta}$,
which approaches 1 as $n\to\infty$.

%Given $\delta$, we assume that $n$ be large enough so that 
%$a_{n-1}\le (1+\frac \delta 2)^{1/(1+\beta)}a_n$ 
%(we impose additional lower bounds on $n$ later in the proof). 
%This implies that $x(2x)^\gamma\in [a_n^\gamma,(1+\frac \delta2)a_n^\gamma]$
%for each $x\in I_n$ and $\gamma\in[\alpha,\beta]$. 

This allows us to compare the displacement of $x$ under the composition 
$f_{\gamma_k}\circ\ldots\circ f_{\gamma_1}$
(that is $f_{\gamma_k}\circ\ldots\circ f_{\gamma_1}(x)-x$) to 
$\sum_{j=1}^k a_n(2a_n)^{\gamma_j}$, the sum of the displacements
$f_{\gamma_j}(a_n)-a_n$, provided that $f_{\gamma_j}\circ\ldots\circ f_{\gamma_1}(x)$
remains in $I_n$ for $j=0,\ldots,k-1$. It is convenient (to maximize the strength
of the conclusion from Hoeffding's inequality) to rescale the displacements by a multiplicative factor of 
$2^{-\alpha} a_n^{-(1+\alpha)}=2^{-\alpha} e^{(1+\alpha)\sqrt n}$.

Let $(X_i)_{i\in\N}$ be a sequence of i.i.d.~random variables, where 
$X_i=(2a_n)^{\gamma_i-\alpha}$ and $\gamma_i$ is distributed according to $\nu$. 
In particular, $0\le X_i\le 1$ for each $i$. Similarly, set $Y_i=(2a_{n-1})^{\gamma_i-\alpha}$,
so that for $x\in I_n$, the step size $x(2x)^{\gamma_i}$ satisfies $2^\alpha a_n^{1+\alpha}X_i
\le x(2x)^{\gamma_i} \le 2^\alpha a_{n-1}^{1+\alpha} Y_i$. 
We use $\mg P$ and $\mg E$ to refer to the distribution on the $(X_i)$ and $(Y_i)$
in order to distinguish from $\mb P$ used elsewhere in the paper. Notice also that
$N_n=(a_{n-1}-a_n)/[2^\alpha a_n^{1+\alpha}\mg E X]$ (we drop the $n$ subscript in the following computations and note that the random variables here are $X= (2a_n)^{\gamma-\alpha}$ and $Y=(2a_{n-1})^{\gamma-\alpha}$).

If we define (dependent) random variables 
$D_k$ by $D_k = f_{\gamma_k}\circ\ldots\circ f_{\gamma_1}(x)- f_{\gamma_{k-1}}\circ \ldots f_{\gamma_1}(x)$, then 
provided $ f_{\gamma_k}\circ\ldots\circ f_{\gamma_1}(x)$ remains below $a_{n-1}$ we have 
$D_1+\ldots+ D_k\le Y_1+\ldots+Y_k$. Hence for $x\in [e^{-\sqrt n},f_\alpha(e^{-\sqrt n})]$, 
$2^\alpha \alpha_{n-1}^{1+\alpha}(Y_1+\ldots+Y_{(1-\eta)N})<a_{n-1}-f_\alpha(a_n)$
implies $f_{\gamma_{(1-\eta)N}}\circ 
\ldots f_{\gamma_1}(x)<a_{n-1}$. For $x\in [e^{-\sqrt n},f_\alpha(e^{-\sqrt n}))$, we have
%{\color{blue} Recall that $x < f_{\alpha}(a_n)$, 
%therefore the minimum displacement $x$ has to undergo in order to escape $I_n$ is 
%$a_n-f_{\alpha}(a_n)$, and the probability that the actual displacement after $(1-\eta)N$ steps 
%is more than $a_{n-1}-f_\alpha(a_n)$ can be upper bounded with the probability that 
%$2^\alpha a_{n-1}^{1+\alpha}(Y_1+\ldots+Y_{(1-\eta)N})$ is larger than $a_{n-1}-f_{\alpha}(a_n)$ giving}
\begin{align*}
&\mg P(f_{\gamma_{(1-\eta)N}}\circ\ldots\circ f_{\gamma_1}(x)>a_{n-1})\\
&\le \mg P\left(2^\alpha a_{n-1}^{1+\alpha}(Y_1+\ldots+Y_{(1-\eta)N})>
a_{n-1}-f_\alpha(a_n)\right)\\
&=\mg P\left(\frac{Y_1+\ldots+Y_{(1-\eta)N}}{(1-\eta)N}-\mg E Y>
\left(\frac{a_{n}}{a_{n-1}}\right)^{1+\alpha}\frac{(a_{n-1}-f_\alpha(a_n))\mg E X}
{(1-\eta)(a_{n-1}-a_n)}-\mg E Y\right)\\
&\le \mg P\left(\frac{Y_1+\ldots+Y_{(1-\eta)N}}{(1-\eta)N}-\mg E Y
>\tfrac{\eta} 2\mg E X\right)\\
&\le e^{-2(1-\eta)N(\frac{\eta} 2\mg E X)^2},
\end{align*}
The equality follows from the expression for $N_n$ and a simple manipulation 
of the inequality in parentheses.
%; the random variable $Y=(2a_{n-1})^{\gamma-\alpha}$. 
The second inequality holds for all sufficiently large $n$ using the facts 
$a_n/a_{n-1}=1+o(1)$, $(a_{n-1}-f_\alpha(a_n))/(a_{n-1}-a_n)=1+o(1)$ and
$\mg E X=(1+o(1))\mg E Y$, which follows from the comparison 
of step sizes mentioned above; the final inequality is an application from Hoeffding's
inequality. Since $N\ge(a_{n-1}-a_n)/(2^\alpha a_n^{1+\alpha})$,
we see that $N\ge e^{\alpha\sqrt n}/(2^{1+\alpha}n)$, while $\gamma-\alpha$ takes values in the range
$[0,\frac\alpha 4]$ with positive probability, $p$ say, so that 
$(\mg E X)^2>p^2e^{-\frac \alpha 2\sqrt n}$.
Thus we see the probability of spending less than $(1-\eta)N$ steps in $I_n$ is (much) smaller than
$2^{-n-1}$ for large $n$.

An analogous argument based on studying the probability that
$2^\alpha a_n^{1+\alpha}(X_1+\ldots+X_{(1+\eta)N})<a_{n-1}-a_n$ shows that the probability
of spending more than $(1+\eta)N$ steps in $I_n$ is smaller than $2^{-n-1}$ for large $n$. 
\end{proof}

\begin{proof}[Proof of Lemma \ref{lem:rettimeest}]
Let $\delta>0$ be fixed. 
Given $n>0$, let $x\in [e^{-\sqrt n},e^{-\sqrt{n-1}})$, $m=\lceil \alpha^2/(4\beta^2)n\rceil$
and $p=\lfloor n-\delta\sqrt n\rfloor$. Notice that $a_p\le (1+\delta)x$ for large $n$. 
Let $E$ be the event that starting from $a_p$, the number of steps spent in
$I_j$ is at most $e^{\delta\alpha/9}N_j$ for each $j=m+1,\ldots,p$.
Applying Lemma \ref{lem:Intime} with $\eta$ taken to be $e^{\delta\alpha/9}-1$,
provided $n$ (and hence $m$) is sufficiently large, $E$ has probability
at least $1-2^{-m+1}$. After leaving $I_{m+1}$, the position exceeds
$e^{-\sqrt m}$ and the number of steps before hitting $Y$ is bounded above by $e^{\beta\sqrt m}$,
so that the contribution to $\phi_Y$ from this part of the orbit is bounded above by 
$e^{\beta\sqrt m}\|\phi\|$. Similarly, for $\bo\omega\in E$, 
the contribution from $I_p\cup\ldots\cup I_{m+1}$ 
is bounded above by $e^{\delta\alpha/9}
\sum_{j=m+1}^{p}N_j\big(\phi(a_j)+\|\phi\|_\text{Lip}(a_{j-1}-a_j)\big)$.
We claim that for all sufficiently large $n$ (with $m$ and $p$ related to $n$ as above),
\begin{equation}\label{eq:toshow}
e^{\delta\alpha/9}
\sum_{j=m+1}^{p}N_j\Big(\phi(a_j)+\|\phi\|_\text{Lip}(a_{j-1}-a_j)\Big)+e^{\beta\sqrt m}\|\phi\|
<\int_x^{\frac 12}\frac{\phi(t)}{\int t(2t)^\gamma\,d\nu(\gamma)}\,dt.
\end{equation}
One may check that $\int (2t)^\gamma\,d\nu(\gamma)=(2t)^{\alpha+\epsilon(t)}$, where 
$\epsilon(t)\to 0$ as $t\to 0$. In particular, the right side of \eqref{eq:toshow} exceeds 
$Ce^{\alpha\sqrt n}$ for all large $n$ (and $x\in I_n$). By the choice of $m$, $e^{\beta\sqrt m}\|\phi\|
=o(e^{\alpha\sqrt n})$.
Similarly, for all sufficiently large $n$, $N_j\le e^{(\alpha+\frac 12)\sqrt j}$ for 
all $j\ge m$. Also $a_{j-1}-a_j=o(e^{-\sqrt j})$,
so that $\sum_{j=m+1}^p N_j(a_{j-1}-a_j)=o(ne^{(\alpha-\frac12)\sqrt n})$.
To show \eqref{eq:toshow} it therefore suffices to show
\begin{equation}\label{eq:toshow2}
\sum_{j=m+1}^p N_j\phi(a_j)\le e^{-\delta\alpha/9}
\int_x^{\frac 12} \frac{\phi(t)}{\int t(2t)^{\gamma}\,d\nu(t)}
\,dt.
\end{equation}
To see this, set $h(t)=\phi(t)/\int t(2t)^\gamma\,d\nu(t)$ and observe
$$
\sum_{j=m+1}^p N_j\phi(a_j)=
\sum_{j=m+1}^p (a_{j-1}-a_j)h(a_j)
$$
This Riemann sum is 
$$
(1+o(1))\int_{a_p}^{a_m}h(t)\,dt,
$$

Notice that provided $n$ is sufficiently large, $a_p\ge e^{\delta/2}x$. 
We also see  that for any $b>1$, for
all sufficiently small $t$, $h(bt)\le \frac{b^{\alpha/3}}{b^{\alpha+1}}h(t)$,
so that for large $n$ (and corresponding $m$ and $p$),
\begin{align*}
\int_{a_p}^{a_m} h(t)\,dt&= e^{\delta/2}\int_{e^{-\delta/2}a_p}^{e^{-\delta/2}a_m}h(e^{\delta/2}t)\, dt
\\
&\le e^{-\delta\alpha/3}\int_{e^{-\delta/2}a_p}^{e^{-\delta/2}a_m}h(t)\,dt\\
&\le e^{-\delta\alpha/3}\int_x^{e^{-\delta/2}a_m}h(t)\,dt.
\end{align*}
Since 
$$
\int_{e^{-\delta/2}a_m}^{\frac 12}h(t)\,dt=o\left(\int_x^{\frac 12} h(t)\,dt\right),
$$
the claimed inequalities \eqref{eq:toshow2} and hence \eqref{eq:toshow} follow. 
We have therefore shown that for $\bo\omega\in E$, $\phi_Y(\bo\omega,a_p)<q(x)$.
Now if $z\ge (1+\delta)x$, since $z\ge a_p$, the same bound applies to $\phi_Y(\bo\omega,z)$.
Hence we have established the second statement of the lemma. The first statement is proved in a very similar way, by using Lemma \ref{lem:Intime}
to give lower bounds on the time spent in $I_q,\ldots,I_m$ where $q=\lceil n+\delta\sqrt n\,\rceil$.
\end{proof}

\begin{proof}[Proof of Proposition \ref{prop:phiYinbasin}]
Since for $x\in [\frac 12,1]$, $f_\gamma(x)=2x-1$ for all $\gamma$, 
Lemma \ref{lem:rettimeest} shows that $\PP_Y(\{\bo\omega,x)\colon 
\phi_Y(\bo\omega,x)>t\})=q^{-1}(t)(1+o(1))$. It therefore suffices to show that $q^{-1}(t)$
is a regularly varying function with order $-1/\alpha$. However, Theorem 1.5.12 of
\cite{Bingham} shows that if $g(x)$ is regularly varying of order $\alpha$ as $x\to \infty$,
then $g^{-1}(x)$ is regularly varying of order $1/\alpha$. Applying this to $g(t)=q(1/t)$, it suffices
to show that $\lim_{x\to\infty}g(bx)/g(x)=b^{\alpha}$ for each $b>0$. A calculation shows
\[
g(x)=\int_2^x \frac{1}{u\mg \int(\frac 2u)^\gamma\,d\nu(\gamma)}\,du.
\]
Now set $h(s)=\int e^{-s\gamma}\,d\nu(\gamma)$. If $s=(1-\lambda)r+\lambda p$,
noticing that $\frac{1}{1-\lambda}$ and $\frac 1\lambda$ are conjugate H\"older exponents, 
we have
\begin{align*}
h(s)&=\int e^{-(1-\lambda) r\gamma}e^{-\lambda p\gamma}\,d\nu(\gamma)\\
&\le \left(\int e^{-r\gamma}\,d\nu(\gamma)\right)^{1-\lambda}
\left(\int e^{-p\gamma}\,d\nu(\gamma)\right)^{\lambda}\\
&=h(r)^{1-\lambda}h(p)^{\lambda},
\end{align*}
so that $h$ is log-convex. Since for all $\epsilon>0$, 
$\nu([\alpha,\alpha+\epsilon])e^{-(\alpha+\epsilon) s}\le h(s)\le e^{-\alpha s}$ for large $s$,
it follows that $(\log h)'(s)\to-\alpha$ as $s\to\infty$
and $h(s+l)/h(s)\to e^{-\alpha l}$ as $s\to\infty$. 
It follows that $u\mapsto \int (\frac 2u)^\gamma\,d\nu(\gamma)=\int e^{-\log(u/2)\gamma}d\nu(\gamma)$ is regularly varying with 
index $-\alpha$ as $u\to\infty$, in fact 
\[
\frac{\int e^{-\log(\lambda u/2)\gamma}d\nu(\gamma)}{\int e^{-\log( u/2)\gamma}d\nu(\gamma)}=\frac{h(\log(u/2)+\log\lambda)}{h(\log(u/2))}\rightarrow e^{-\alpha\log\lambda}=\lambda^{-\alpha}.
\] Now it follows from Theorem 1.5.11 (i)  of \cite{Bingham} that since $[h(\log(u/2))] ^{-1}$ is regularly varying with index $\alpha$, for $\sigma=-1\ge -(\alpha+1)$
\[
\frac{x^{\sigma+1}[h(\log(x/2))]^{-1}}{\int_2^xu^{-\sigma}[h(\log(u/2))]^{-1}du}=\frac{[h(\log(x/2))]^{-1}}{g(x)}\rightarrow \sigma+\alpha+1=\alpha.
\] 
and therefore $g(x)$ is regularly varying with index $\alpha$ as required.
\end{proof}

\begin{proof}[Proof of Theorem \ref{Thm:LimitUnfold}]
We can restrict to the case $\int \phi\, d\pi(x)=0$, where $\pi$
is the stationary measure. By Proposition \ref{prop:phiYinbasin},
the cumulative distribution function, $G_\phi$ of $\phi_Y$ satisfies
$1-G_\phi(t)=L_\phi(t)t^{-\frac 1\alpha}$, where $L_\phi$ is slowly varying. 
We showed in the proof of Lemma \ref{lem:rettimeest} that 
$1-G_\phi(t)=(1+o(1))q_\phi^{-1}(t)$ where
$q_\phi(x)=\int_x^{\frac 12}\phi(z)/[\int z(2z)^\gamma\,d\nu(\gamma)]\,dz$.
The same proposition applies if $\phi$ is replaced by $\mathbf 1$ so that 
the cumulative distribution function, $G_\tau$ of $\tau_Y$ satisfies
$1-G_\tau(t)=L_\tau(t)t^{-\frac 1\alpha}$, where $L_\tau$ is again slowly varying. 
As before, $1-G_\tau$ behaves asymptotically as the inverse function of
$q_\tau(x)=\int_x^{\frac 12}1/[\int z(2z)^\gamma\,d\nu(\gamma)]\,dx$.
Since $q_\phi(x)=\phi(0)q_\tau(x)(1+o(1))$ and both are regularly varying of order $-\alpha$, we
see that $q_\phi(\phi(0)^{\frac 1\alpha}x)\sim q_\tau(x)$ and 
$1-G_\tau(t)\sim (1-G_\phi(t))/\phi(0)^{\frac 1\alpha}$.

Let $B_n$ be the solution of $nL_\phi(B_n)=B_n^{\frac1\alpha}$ as in the statement of 
Proposition \ref{prop:inducedstablelaw}. By that proposition, we obtain the existence of $A_n$ and $A_n'$
such that $(S_Y^n\phi_Y-A_n)/B_n$ and $(S_Y^n\tau_Y-A_n')/B_n$ converge to stable laws.  By Kac's Lemma, we may assume $A_n' = n/\pi(Y)$. 

Since $(S_Y^n\tau_Y-n/\pi(Y))/B_n$ converges in
distribution, then it is also tight.  Since $\phi_Y$ is integrable with
$\int_Y \phi_Yd\pi_Y= \int \phi d\hat \pi =0$, where $\hat \pi$ is the unfolding of $\pi_Y$ to $[0,1]$ and $\hat \pi$ equals, up to a normalizing factor, the invariant measure $\pi$. By Birkhoff's ergodic theorem $S_Y^{B_n}\phi_Y/{B_n}$
tends to zero almost everywhere. This implies that $\max_{0\leq k\leq
  B_n}|S_Y^k\phi_Y|/B_n$ tends to zero almost surely (and thus in
probability). To see this, consider $\mc A$ the set where the Birkhoff averages
have limit zero. For $x\in\mc A$, call $k_n(x)$ the sequence of
indices realizing the maxima. The sequence $(k_n(x))$ is non-decreasing.
If $k_n(x)$ is eventually constant, then
$|S_Y^{k_n(x)}\phi_Y(x)|/B_n\rightarrow 0$. If not, then
\[
\frac{|S_Y^{k_n(x)}\phi_Y(x)|}{B_n}=\frac{k_n(x)}{B_n}
\frac{|S_Y^{k_n(x)}\phi_Y(x)|}{k_n(x)}\rightarrow
0
\]
By Theorem \ref{Thm:GouezUnfold} the conclusion of the theorem follows when $\phi(0) >0$.
\end{proof}

Our final step it to consider the case $\phi(0)<0$. One simply observes that the statement of the Theorem \ref{Thm:GouezUnfold} is invariant under $\phi \rightarrow -\phi$, $A'_n \rightarrow -A'_n$ and $\mc Z  \rightarrow -\mc Z$, another stable law, and apply the previous argument.  

\section{Diffusion Driven Decay of Correlations}\label{sec:diffus}

In this section we consider a family of distributions $\nu$ on the parameter
space, with unbounded support. More precisely, for $\alpha,\epsilon>0$, $\nu_{\alpha,\epsilon}$ is
supported on $[\alpha,\infty)$, and given by 
\[
\nu_{\alpha,\epsilon}[0,t]=F(t):=
\begin{cases}1-\left(\frac t\alpha\right)^{-\epsilon}&\text{if $t\ge\alpha$;}\\
0&\text{if $t<\alpha$.}
\end{cases}
\]
Notice that since the measure has unbounded support, the considerations made in sections \ref{Sec:EstReturn} and \ref{Sec:Renewal} on the uniform bound of the distortion, do not apply. Furthermore, the fat polynomial tails of the distribution of $\nu$, imply that arbitrarily large parameters are sampled somewhat frequently.

\subsection{Markov chains with subgeometric rates of convergence}
The arguments in sections \ref{Sec:EstReturn} and \ref{Sec:Renewal} do not apply to this setup. Instead, we exploit the absolute continuity of $\nu$ to translate the deterministic process \eqref{Eq:ContSkew} into a
Markov process on the state space $[0,1]$ and apply  a
theorem on Markov chains (see \cite{tuominen1994subgeometric}) to
prove the existence of a stationary measure and to estimate the
asymptotic rate of correlation decay. To this end, consider the
stationary Markov process $(X_t)_{t\in\N_0}$ on the probability
space $(\Omega\times [0,1],\mc B(\Omega\times [0,1]),\mb P)$ taking
values in $[0,1]$, where $X_t(\bo\omega,x_0)=\Pi \circ
F^t(\bo\omega,x_0)$. Calling $\Gamma_x:[\alpha,\beta]\rightarrow
[0,1]$ the map $\Gamma_x(\omega)=f_{\omega}(x)$, then the transition
probability $P(x,A)=\mb P\left(X_t\in A|\;X_{t-1}=x\right)$, is given
by $P(x,\cdot)=(\Gamma_x)_*\nu$, i.e. the push-forward under
$\Gamma_x$ of the probability $\nu$ on the parameter space, or more
concretely $P(x,B)=\nu\{\gamma\colon f_\gamma(x)\in B\}$. 
The Markov chain defines an evolution of measures, which we denote by $Q$:
$Q\mu(B)=\int_{[0,1]}P(x,B)\,d\mu(x)$, that is if $X_n$ has distribution $\mu$, then
$X_{n+1}$ has distribution $Q\mu$. In case $\mu$ is an absolutely continuous measure,
with density $\rho$, $Q\mu$ is again absolutely continuous with density
\begin{equation}\label{eq:dens-evolv}
P\rho(x)=\int \frac{\rho({f_\gamma|_{[0,\frac 12]}}^{-1}(x))}
{f_\gamma'({f_\gamma|_{[0,\frac 12]}}^{-1}(x))}\,d\nu(\gamma)+\frac{\rho(\frac {x+1}2)}2.
\end{equation}

For $A\subset [0,1]$, we define
$\tau_A(\bo \omega,x)=\inf\{n\in\N|\; X_n(\bo\omega,x)\in A \}$,  and let us call $\mb P_x$, the probability measure $\mb P$ conditioned on $\{X_0=x\}$.

\begin{definition}
A Markov chain is $\psi-$irreducible if there is a measure $\psi$ such
that for every $A\in\mc B(S)$ with $\psi(A)>0$ and every $x\in S$
\[
\mb P_x\left(\tau_A<\infty\right)>0.
\]
\end{definition}

\begin{lemma}
The Markov chain $\{X_t\}_{t\in\N_0}$ is $\psi-$irreducible and aperiodic.
\end{lemma}
\begin{proof}
The only problem with these two properties is that
$P(0,\cdot)=\delta_0$. However, this can be fixed by removing the set
$\{0\}\cup\{h^{-n}(\tfrac 12):n\ge 0\}$ from the state space,
where $h:[\frac12,1]\rightarrow [0,1]$ is $h(x)=2x-1$ mod $1$, and
redefining the $\sigma-$algebras and transition probabilities accordingly.  In this modified state space, every orbit eventually enters $(0,1/2)$ and is spread by the diffusion. This implies that any set of positive Lebesgue measure is  eventually visited by a random orbit originating from any given $x$. Aperiodicity is a consequence of the two branches of all the maps $f_{\gamma}$ being onto.
\end{proof}

\begin{definition}\label{Def:Petiteset}
Given a probability distribution $a:\N\rightarrow\R^+_0$, and a
nontrivial measure $\nu_a$, a set $C\in\mc B(S)$ is $\nu_a$-petite if
\[
K_a(x,\cdot):=\sum_{n=1}^{\infty}a(n)P^n(x, \cdot)\ge
\nu_a(\cdot),\quad\quad\forall x\in C.
\]
\end{definition}

\begin{definition} 
A function $r:\N_0\rightarrow \R$ is a subgeometric rate function if
there is a non-decreasing $r_0:\N_0\rightarrow \R$ with $r_0(1)\ge 2$
and $\log (r_0(n))/n\rightarrow 0$ as $n\rightarrow\infty$ such that
\[
\liminf_{n\rightarrow \infty}\frac{r(n)}{r_0(n)}>0,\quad\quad
\limsup_{n\rightarrow\infty}\frac{r(n)}{r_0(n)}<\infty.
\]
\end{definition}

The main result we are going to use in this section is the following
theorem on subgeometric rates of convergence of ergodic Markov chains.
\begin{theorem}[\cite{tuominen1994subgeometric}]\label{Thm:TT}
Suppose that $\{X_t\}_{t\in\N_0}$ is a discrete time Markov process on
a general state space $S$ endowed with a countably generated
$\sigma$-field $\mc B(S)$ which is $\psi$-irreducible and
aperiodic. Suppose that there is a petite set $C\in \mc B(S)$  and a subgeometric rate function
$r:\N\rightarrow\R^+$ (see
above for the definitions) such that
\begin{equation}\label{Eq:CondHitPet}
\sup_{x\in C}\mb E_x\left[\sum_{n=0}^{\tau_C-1}r(n)\right]<\infty.
\end{equation}
where $\tau_C:=\inf\{t\in\N:\; X_t\in C\}$, and $\mb E_x$ is the
expectation with respect to the probability law of the Markov process
conditioned on $X_0=x$.  Then the Markov process admits a stationary
measure $\pi$, and for almost every point $x\in S$
\[
\lim_{n\rightarrow\infty}r(n)\|P^n(x,\cdot)-\pi\|_{TV}=0,
\]
where $\| \cdot \|_{TV}$ denotes the total variation norm difference between the two measures. 
Furthermore, if \eqref{Eq:CondHitPet} holds and $\lambda$ is a
probability measure on $S$ that satisfies
\begin{equation}\label{Eq:CondHitMeas}
\mb E_\lambda\left[\sum_{n=0}^{\tau_C-1}r(n)\right]<\infty,
\end{equation}
then 
\[
r(n)\int\|P^n(x,\cdot)-\pi\|_{TV}d\lambda(x)\rightarrow 0.
\]
\end{theorem}

For the remainder of this section, let $\alpha$ and $\epsilon$ be fixed.
Let $b<\frac 12$ be chosen so that $f_\alpha(b)>\frac{3+b}4$.

We compute the one
step probability transition density starting from a point $x<\frac
12$, which will be denoted $p_x(y)$ (with support
$[x,f_\alpha(x)]$).  For $x<y\le f_\alpha(x)$, define $\gamma_x(y)$ to be the
value of $\gamma$ such that $f_\gamma(x)=y$.  That is $\gamma$
satisfies $x(1+(2x)^\gamma)=y$, or $(2x)^\gamma=\frac{y-x}x$,
giving an explicit form:
$\gamma_x(y)=\log\frac{x}{y-x}/\log\frac{1}{2x}$, where we took
reciprocals of both numerator and denominator to ensure positivity
of the logarithms. Note that this is a decreasing function of
$y$. We can then obtain $p_x(y)$ by
\begin{align*}
p_x(y)&=\lim_{h\to
  0}\frac{F(\gamma_x(y))-F(\gamma_x(y+h))}h\\ &=-F'(\gamma_x(y))\cdot
\gamma_x'(y)\\ &=\frac{\epsilon\alpha^\epsilon}{(\gamma_x(y))^{1+\epsilon}}\cdot
\frac{1}{(y-x)\log(\frac{1}{2x})}\\ &=\frac{\epsilon\alpha^\epsilon\big(\log(\frac
  1{2x})\big)^\epsilon}
     {\big(\log(\frac{x}{y-x})\big)^{1+\epsilon}(y-x)}.
\end{align*}
Recall that $p_x(\cdot)$ is supported on $[x,f_\alpha(x)]$.  By the
choice of $b$ and since $b<x<\frac 12$, this contains the sub-interval
$[\frac{b+1}2,\frac{b+3}4]$.   We will denote by $p^n_x(y)$ the $n-$step probability transition density.

In this case where $\nu$ has a power law distribution, 
we can write the transition operator \eqref{eq:dens-evolv} of the Markov process in the form:
\[
P\rho(y)=\int_{0}^{\frac{1}{2}}p_x(y)\rho(x)dx+
\tfrac{1}{2}\rho\big(\tfrac{y+1}{2}\big),
\]
where $p_x(y)$ is taken to be 0 if $y\le x$ or $y>f_\alpha(x)$.

\begin{remark}
For any fixed $b<\frac 12$, the set 
\[
C:=[f_\alpha^{-1}(b),b]
\]
is a petite set. In the case where $f_\alpha(b)>\frac{b+3}4$ that we are considering, 
it is not hard to see that $p^2_x(y)$ is uniformly bounded below for $x\in C$ and 
$y\in [\frac 12,\frac{b+1}2]$. Therefore, in Definition \ref{Def:Petiteset} one can pick $a(n)=1$ for $n=2$ and zero otherwise, with  $\nu_a$  a multiple of the Lebesgue measure on $[\frac 12,\frac{b+1}2]$.
\end{remark}

\subsection{Outline of the proof of Theorem \ref{thm:Linftydecorr}} The rest of the section is mostly dedicated to estimating the return times to $C$ and to finding  rates $r$ for which \eqref{Eq:CondHitPet} holds. 

   Notice that  $\frac{b+1}2$ is the preimage of $b$ under the second branch. Two facts play an important role in the arguments that follow: by the particular choice of $C$, any random orbit starting from $(0,b)$ will pass through $C$ before landing to the right of it; points in $H:=[\frac 12,\frac{b+1}2)$ are mapped to $(0,b)$, i.e. to the left of the petite set $C$. Any point in $(b,1)$ must visit $H$ before hitting $C$.

\begin{figure}[h]
\begin{center}
\includegraphics[width=2.5in]{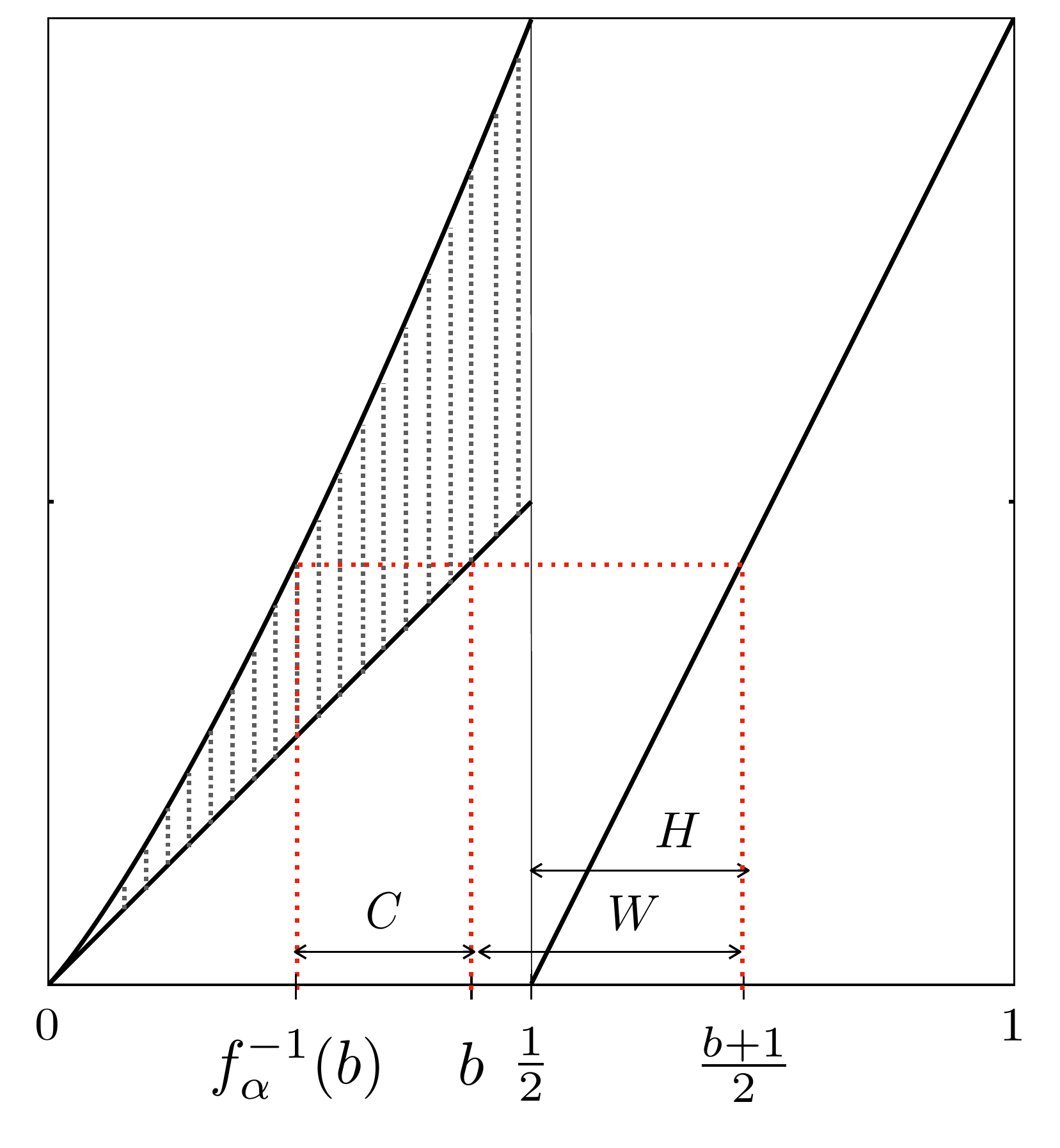}
\end{center}
\caption{Schematic indicating the petite set $C$ and the subintervals $W$ and $H$.}
\end{figure}

\emph{Step 1.} First we show that random orbits originating from $C$ tend to end up to the left of $C$ quite fast. More precisely, given any
initial condition to the right of $b$, the probability that a random orbit
hasn't entered $H$ by time $t$ is exponentially small in $t$. This is the content of Proposition \ref{Prop:StretExp} and Lemma \ref{Lem:LeaveC} below which need several other lemmas to be proven. 

\emph{Step 2.} Then we study the hitting times to $C$ for orbits originating on $(0,b)$. For $x$ close to 0, 
by Lemma 5.7 the time to hit $C$ is of the rough order $x^{-\alpha}$. Since for $z\in H$, 
$f_\gamma(z)=2z-1<b$ for all $\gamma$, the distribution of times to enter $C$ after hitting 
$H$ is determined by the distribution of positions at which $H$ is hit. Lemma 6.15 assembles the 
prior facts to show that $X_{\tau_H}$, conditioned on starting at 
$x\in C$ is absolutely continuous with density bounded above uniformly in $x$.
This ensures that the distribution of the time between hitting $H$ and entering $C$
has a polynomial tail which, compared to the exponential tails of the times estimated in Step 1, dominate the statistics of the returns to $C$. In Proposition \ref{prop:Eretpower} we put together all the estimates and obtain subgeometric rates  $r$ for which \eqref{Eq:CondHitPet} holds. 

\emph{Step 3.} Finally, we  apply Theorem \ref{Thm:TT}  to prove Theorem \ref{thm:Linftydecorr}. 

\subsection{Step 1}

In order to show that the hitting times to  $H$ have exponential tails, we divide $(b,1)$ into the intervals $W:=[b,\frac{b+1}2)$ and $[\frac{b+1}2,1)$. Notice that $W$ is such that for each $y\in
[\frac{b+1}2,1)$, the iterates of $y$ under the $f_\gamma$'s remain in
 the right branch until they hit $W$. We are going to show that
 \begin{itemize}
 \item[i)] one can control the density of the random orbits originating from $C$ at the stopping time $\tau_W$, that is the first entry to $W$ (lemmas \ref{Lem:BoundsW}, \ref{Lem:LeaveC});
 \item[ii)] after each return to $W$, at least a fixed fraction of the random orbits will enter $H$ and then be mapped to the left of the petite set (Lemma \ref{Lem:ReturnstoW});
 \item[iii)] the times between consecutive entries to $W$ before hitting $H$ have at most exponential tails (Lemma \ref{Lem:EstReturnW2}).
 \end{itemize}

\begin{lemma}\label{Lem:Bounds}
Let $b<\frac 12$ satisfy $f_\alpha(b)>\frac{b+3}4$.
  There is $r>0$ such that for every $x\in
  (b,\frac{1}{2})$ \[r(x):=\int_{(\frac{b+1}{2},1)}p_x(y)dy>r\] and
  $r\rightarrow 1$ when $b\rightarrow \frac{1}{2}$.
\end{lemma}

\begin{proof}
Notice that for $x\in [b,\frac 12)$, $r(x)=F(\gamma_x(\frac{b+1}2))$, which is strictly positive,
and approaches 1 as $x\to \frac 12$. 
\end{proof}

 We show that there exists  $C_\#>0$ such that for all $x\in
(b,\frac{1}{2})$, conditioned on $X_0=x$ and $X_1>\frac{b+1}2$ (so
that $X_2>b$) then the distribution of $X_{\tau_W}$ has density
bounded between $C_\#^{-1}$ and $C_\#$, where $\tau_W$ is the time of the first
return to $W$.   To prove this, we first need an estimate on $p_x$:

\begin{lemma}\label{Lem:Est1}
There exists $K$ such that for all $x\in(b,\frac 12)$, and all $y,y'\in
[\frac{b+1}2,f_\alpha(x)]$ satisfying $2y-1\le y'< y$ (that is, for
any $\gamma$, $f_\gamma(y)\le y'<y$), one has $p_x(y')/p_x(y)$ is
bounded between $\frac{1}{K}$ and $K$.
\end{lemma}

\begin{proof}
We have
$$
\frac{p_x(y)}{p_x(y')}=\frac{\big(\log(\frac x{y'-x})\big)^{1+\epsilon}(y'-x)}
{\big(\log(\frac x{y-x})\big)^{1+\epsilon}(y-x)}.
$$
Clearly the ratio $(y'-x)/(y-x)$ is uniformly bounded above and below for $x\le \frac 12$ and 
$y,y'\in[\frac{b+1}2,1]$ so it suffices to establish $\log(\frac x{y'-x})/\log(\frac x{y-x})$
is uniformly bounded above and below for $x$, $y$ and $y'$ as in the statement of the claim.
It is clear that the numerator exceeds the denominator, so it suffices to give an upper bound.

There exist positive numbers $a$ and $A$ such that $a(2-u)\le \log \frac1{u-1}\le A(2-u)$ for all 
$u\in [b+1,2]$. Applying this with $u$ taken to be $y'/x$ and $y/x$, we see that
$\log(\frac x{y'-x})/\log(\frac x{y-x})\le A(2-\frac {y'}x)/a(2-\frac yx)$, so that
it suffices to give a uniform upper bound for $(2x-y')/(2x-y)$.
Since $y'\ge 2y-1$, we see $(2x-y')/(2x-y)\le (1+2x-2y)/(2x-y)=2+(1-2x)/(2x-y)$.

Finally $2x-y\ge 2x-f_\alpha(x)=x(1-(2x)^\alpha)\ge b(1-(2x)^\alpha)
\ge \alpha b(1-2x)$, giving the required upper bound for $(2x-y')/(2x-y)$. 
\end{proof}

\begin{lemma}\label{Lem:BoundsW}
There exists $C_\#>1$ such that for all $x\in (b,\frac 12)$ the density
$\rho_{x,W}(\cdot)$ of $X_{\tau_W}$ conditioned on $X_0=x$, $X_1>\frac
{b+1}2$ satisfies $\frac 1C_\#\le \rho_{x,W}(y)\le C_\#$ for all $y\in W=[b,\frac{b+1}2)$.
\end{lemma}

\begin{proof}
We establish this by showing that there exists $K>0$ such that for all
$x\in (b,\frac12)$ and all $z,z'\in [b,\frac {b+1}2]$,
$\rho_{x,W}(z)/\rho_{x,W}(z')\le K$.

We first observe that
\begin{equation}\label{Eq:CollapsetoW}
\rho_{x,W}(z)=\sum_{n=1}^\infty 2^{-n}p_x\left(\frac{2^n-1+z}{2^n}\right)
/\mathbb P_x(X_1>\tfrac{b+1}2).
\end{equation}
However, since $p_x$ is supported on $[x,f_\alpha(x)]$, there are
only finitely many non-trivial terms in the sum. Also if $z<z'$ both
belong to $[b,\frac{b+1}2]$, then $z'\le \frac{z+1}2$ so
$$
\frac{2^{n}-1+z}{2^{n}}
< \frac{2^{n}-1+z'}{2^{n}}
\le \frac{2^{n+1}-1+z}{2^{n+1}}.
$$
Hence the number of non-trivial terms in the summation for
$\rho_{x,W}(z)$ is at least the number of terms in the summation for
$\rho_{x,W}(z')$ and at most one more.

Now suppose that $n$ is the largest number such that
$\frac{2^{n}-1+z'}{2^{n}}\le f_\alpha(x)$.  Since
$\frac{2^{n}-1+z}{2^{n}}=2(\frac{2^{n+1}-1+z}{2^{n+1}})-1$, The
previous lemma establishes
$$
p_x\left(\frac{2^{j}-1+z}{2^{j}}\right)\le
K\cdot p_x\left(\frac{2^{j}-1+z'}{2^{j}}\right)\text{\quad for $j=1,\ldots,n$.}
$$
If $\frac{2^{n+1}-1+z}{2^{n+1}}\le f_\alpha(x)$, then 
since 
$$
\frac{2^n-1+z'}{2^n}\le \frac{2^{n+1}-1+z}{2^{n+1}}
<\frac{2^{n+1}-1+z'}{2^{n+1}}
$$
and $\frac{2^n-1+z'}{2^n}=2(\frac{2^{n+1}-1+z'}{2^{n+1}})-1$,
the previous lemma implies 
$$
p_x\left(\frac{2^{n+1}-1+z}{2^{n+1}}\right)
\le K\cdot p_x\left(\frac{2^n-1+z'}{2^n}\right).
$$
Summing these inequalities yields the desired claim.
\end{proof}

The next lemma shows that if the starting density $\rho$ on
$(b,\frac{1}{2})$ is controlled as in equation \eqref{Eq:Condrho}
below, then the condition is invariant under the 
transition operator.
\begin{lemma}\label{Lem:InvClassRho}
Let $\sigma\in(0,1)$. There exists $b_0<\frac 12$ such that for all $b\in (b_0,\frac 12)$
and for every
density $\rho$ on $(b,\frac{1}{2})$ satisfying
\begin{equation}\label{Eq:Condrho}
\rho(x)\le \frac{K}{(x-b)\left(\log\frac{b}{x-b}\right)^{1+\epsilon}},
\end{equation}
then
\[
P\rho(y)\leq \frac{\sigma
  K}{(y-b)\left(\log\frac{b}{y-b}\right)^{1+\epsilon}}
\]
for any $y\in(b,\frac{1}{2})$.

Furthermore, for any $b<\frac 12$, there exists $M$ such that if $\rho$ satisfies \eqref{Eq:Condrho}, then
$P\rho(y)\le M$ for all $y\in [\frac 12,\frac{b+1}2)$. 
\end{lemma}

\begin{proof}
Let $y\le \frac 12$. We start by noticing that for any $0<s<c$
\[
\frac{1}{\left(\log\frac{c}{s}\right)^{1+\epsilon}s}=
\frac{d}{ds}\frac{1}{\epsilon}\left(\log\frac{c}{s}\right)^{-\epsilon}.
\]
Since $p_x(y)>0$ for any $x\in [b,\frac 12)$ and any $x<y$; and
using the hypothesis on $\rho$
\begin{align*}
\int_{b}^{\frac{1}{2}}p_x(y)\rho(x)dx&\le \epsilon\alpha^\epsilon
\int_{b}^{y}
\frac{K\big(\log\tfrac 1{2x}\big)^\epsilon}
{\big(\log\frac{x}{y-x}\big)^{1+\epsilon}(y-x)(x-b)
\left(\log\frac{b}{x-b}\right)^{1+\epsilon}}dx\\
&\le  \epsilon\alpha^\epsilon K\big(\log\tfrac 1{2b}\big)^\epsilon
\int_{b}^{y}
\frac{1}{\big(\log\frac{b}{y-x}\big)^{1+\epsilon}(y-x)(x-b)
\left(\log\frac{b}{x-b}\right)^{1+\epsilon}}dx.\\
\end{align*}
Set $z=\frac{b+y}{2}$ and notice that 
the integrand is symmetric about $z$.
Integrating by parts, the integral becomes
\begin{align*}
I(y):=&-2C(\epsilon)\left[
\left(\log\frac{b}{y-x}\right)^{-\epsilon}\cdot 
\left(\frac{1}{(x-b)(\log\frac{b}{x-b})^{1+\epsilon}}\right)
\right]_z^y\\
&\quad +2C(\epsilon)\int_z^{y}
\left(\log\frac{b}{y-x}\right)^{-\epsilon}\cdot 
\frac{d}{dx}\left(\frac{1}{(x-b)(\log\frac{b}{x-b})^{1+\epsilon}}\right)dx
\\
&\le 2C(\epsilon)
\left(\log\frac{b}{y-z}\right)^{-\epsilon}\cdot 
\frac{1}{(z-b)\left(\log\frac{b}{z-b}\right)^{1+\epsilon}}%\\
%&+2C(\epsilon)\left(\log\frac{b}{y-z}\right)^{-\epsilon}\left[
%\frac{1}{(x-b)(\log\frac{b}{x-b})^{1+\epsilon}}\right]_z^{y}
\\
&\le 4C(\epsilon)\left(\log\frac{b}{y-z}\right)^{-\epsilon} 
\frac{1}{(y-b)\left(\log\frac{b}{y-b}\right)^{1+\epsilon}}\\
&\leq  4K\alpha^\epsilon\big(\log\tfrac 1{2b}\big)^\epsilon
\left(\log\frac{b}{\frac{1-2b}{4}}\right)^{-\epsilon} 
\frac{1}{(y-b)\left(\log\frac{b}{y-b}\right)^{1+\epsilon}},
\end{align*}
where $C(\epsilon)=\alpha^\epsilon K\big(\log\tfrac 1{2b}\big)^\epsilon$
and where, in the first inequality, we used the fact that derivative in the second line is negative
for all $x\in [b,\frac 12]$ for all $b$ sufficiently close to $\frac 12$.
Since  
\[
\lim_{b\rightarrow\frac{1}{2}}4\big(\log\tfrac 1{2b}\big)^\epsilon
\left(\log\frac{b}{\frac{1-2b}{4}}\right)^{-\epsilon}= 0,
\]
the first conclusion follows.

For $y\ge \frac 12$, 
$$
I(y)\le 2\epsilon C(\epsilon)\int_b^{\frac 12}
\frac{1}{\big(\log\frac{b}{y-x}\big)^{1+\epsilon}(y-x)(x-b)
\left(\log\frac{b}{x-b}\right)^{1+\epsilon}}dx.
$$
For any $c\in (b,\frac 12)$, $1/((\log\frac{b}{y-x})^{1+\epsilon}(y-x))$ is uniformly
bounded above for $x\in [b,c]$ and $y\in[\frac 12,\frac{b+1}2]$ and
$1/((x-b)(\log\frac b{x-b})^{1+\epsilon})$ is integrable. 
Similarly, on $[c,\frac 12]$, $1/((x-b)(\log\frac b{x-b})^{1+\epsilon})$ is bounded above
and the functions $1/((\log\frac{b}{y-x})^{1+\epsilon}(y-x))$ for $y\in[\frac12,\frac{b+1}2]$
are  integrable over $[c,\frac 12]$ with integral uniformly bounded in $y$,  so that the second conclusion holds.
\end{proof}

The following lemma shows that the gaps between consecutive entries to $W=[b,\frac{b+1}2)$ 
are dominated by a random variable $Z$ with an exponential tail. 
If $X_0\in W$, we define $\tau_{W,i}$ to be the time of the $i$th re-entry to $W$.
That is $\tau_{W,0}=0$ and
$\tau_{W,i+1}=\min\{n>\tau_{W,i}\colon X_{n-1}\not\in W, X_n\in W\}$. 
 Recall that  $H=[\frac{1}2,\frac{b+1}2)$  is the subset mapped to the left of the petite set. We study the distribution of some random variables
conditioned on the event $\{\tau_{W,1}\le\tau_H\}$. Notice that since $H\subset W$, this
is the event that the process leaves $W$ before first hitting $H$.

\begin{lemma}\label{Lem:EstReturnW2}
There exists an integer-valued random variable $Z$
such that for each absolutely continuous distribution on $(b,\frac 12)$
with density $\rho$ bounded as in \eqref{Eq:Condrho}, there exists a random variable $Y_\rho$ such that
the following properties are satisfied:
\begin{enumerate}
\item If $X_0$ is continuously distributed on $[b,\frac 12)$ with density $\rho$, then
conditional on $\{\tau_{W,1}\le \tau_H\}$, $\tau_{W,1}\le Y_\rho$. \label{it:W<Y}
\item For all $n>0$ and all $\rho$ satisfying \eqref{Eq:Condrho}, \label{it:Z>Y}
$$
\mb P_\rho(Y_\rho\ge n|\tau_{W,1}\le\tau_H)\le \mb P(Z\ge n);
$$
\item There exists $a>0$, such that $\mb E e^{aZ}<\infty$;\label{it:expmoment}
\item Let $X_0$ be absolutely continuously distributed on $[b,\frac 12]$ with density $\rho$ satisfying
\eqref{Eq:Condrho}.
For any $n\ge 2$, conditioned on $\{\tau_{W,1}\le\tau_H\}$; and given that $Y_\rho=n$, the distribution 
of $X_{\tau_{W,1}}$, the position at which the system reenters $W$, is absolutely continuous
with density bounded above and below by constants that do not depend on $\rho$ or $n$.
\label{it:induct}
\end{enumerate}
\end{lemma}

\begin{proof}
Let $\rho$ be a probability density on $[b,\frac 12]$ satisfying \eqref{Eq:Condrho}. 
Let $X_0$ be distributed with density $\rho$. Notice that by Lemma \ref{Lem:Bounds},
the event $\{\tau_{W,1}\le \tau_H\}$, that is that the system leaves $W$ before
entering $H$, has probability bounded below by a constant $r>0$. 
We define $Y$ by
$$
Y_\rho=\begin{cases}
0&\text{if $\tau_H<\tau_{W,1}$;}\\
\tau_{W,1}+\mathsf{Geom}(\tfrac12)&\text{otherwise,}
\end{cases}
$$
where $\mathsf{Geom}(\tfrac 12)$ denotes an independent geometric random variable 
with parameter $\frac 12$, so that conclusion \ref{it:W<Y} is evident.
To establish conclusions \ref{it:Z>Y} and \ref{it:expmoment}, it suffices to show that there 
exists $c>0$ such that for all $\rho$ satisfying \eqref{Eq:Condrho}, one has
$\mb P_\rho(\tau_{W,1}\ge n)\le e^{-nc}$ for all $n$. It then follows that 
there exists $c'>0$ such that $\mb P(Y_\rho\ge n)\le e^{-nc'}$ for all $n$.
Then $Z$ can be defined by $Z=n$ with probability 
$e^{-nc'}/r$ for all $n>n_0$ where $n_0$ is chosen so that $\sum_{n>n_0}e^{-nc'}/r<1$;
and $Z=n_0$ with probability $1-\sum_{n>n_0}e^{-nc'}/r$.

Notice that in order that $\tau_{W,1}\ge 3n$, at least one of the following must occur: 
the system must remain in $W$ for $n$ steps; the system must exit $W$
to a point above $(b+2^{n}-1)/2^{n}$ (so that it takes $n$ or more steps to re-enter $W$);
or the geometric random variable must take a value of $n$ or above.
Then $\mb P(\tau_{W,1}\ge 3n)$ is at most the sum of these three probabilities. 
The first of these has probability at most $(1-r)^{n}$ by Lemma \ref{Lem:Bounds}.
The third event has probability $2^{-n}$.

For the second event, note that for $x\in [b,\frac12]$, it is only possible that 
$f_\gamma(x)\ge (b+2^{n}-1)/2^{n}$ if $x>(b+2^{n}-1)/2^{n+1}$,
that is, if $x\in (\frac 12-\frac{1-b}{2^{n+1}},\frac 12)$. 
Define the operator $P_{[b,\frac 12]}$, mapping $L^1([b,\frac{1}2])$ to itself by 
$P_{[b,\frac 12]}f(x)=\mathbf 1_{[b,\frac 12]}(x)Pf(x)$. 
By Lemma \ref{Lem:InvClassRho}, the function 
$g(x)=\sum_{n=1}^\infty P_{[b,\frac 12]}^n(\rho)(x)$ satisfies 
\begin{equation}\label{Eq:sumbound}
g(x)\le \frac 1{1-\sigma}\cdot \frac{K}{(x-b)\left(\log \frac{b}{x-b}\right)^{1+\epsilon}},
\end{equation}
which is bounded above in a neighbourhood of $\frac 12$ by some number $c'$ (which does not depend
on the initial distribution). 
Hence $\mb P(\tau_{[\frac 12-h,\frac 12]}<\tau_{W^c})<hc'$ for all small $h$. 
In particular, the probability of hitting $(\frac 12-\frac{1-b}{2^{n+2}},\frac 12)$ is bounded 
above by a constant multiple of $2^{-n}$, where the constant does not depend on $\rho$. 

To establish conclusion \ref{it:induct}, suppose we are given that $\tau_{W,1}\le \tau_H$
and 	$Y_\rho=n$. We additionally condition on the time taken for the system to exit $W$
and the location of the system prior to exiting $W$. We establish bounds on the density of 
$X_{\tau_{W,1}}$ based on this additional information. Then the bounds without this additional 
conditioning are simply a convex combination of these bounds. 

Thus suppose that the system exits $W$ for the first time at time $k$, 
and we are given that $X_{k-1}=x\in[b,\frac 12]$. Since $Y_\rho=n$, the number of steps 
to reenter $W$ after leaving is one of 1, 2, \ldots, $n-k$. That is $X_k$ belongs to one of
the intervals 
$$
I_j=\left(\frac{b+2^{j}-1}{2^{j}},\frac{b+2^{n-j+1}-1}{2^{n-j+1}}\right],
$$
for $j$ in the range 1 to $\min\{n-k,l\}$ where $l$ is such that 
$f_\alpha(x)\in I_l$. Since $f_\alpha(b)>\frac {b+3}4$, 
we have $l\ge 3$. Since typically
$f_\alpha(x)<\max I_l$, one knows that $X_k$ may only occupy a 
(possibly small, depending on $x$) portion of $I_l$, 
namely $I_l^-=I_l\cap [0,f_\alpha(x))$. 
Hence if $\tau_{W,1}=k+l$, one sees that $X_{\tau_{W,1}}$ is restricted to a 
possibly small sub-interval of $W$ (and therefore its conditional density may not be bounded away from zero). 
This is the reason that we introduced the geometric random variable: to ensure that the return time 
variable, $Y_\rho$ does not completely determine $\tau_{W,1}$ and thereby 
overly constrain the location of $X_{\tau_{W,1}}$.

Conditioned on $X_{k-1}=x$,
the event $\{X_k\ge \frac{b+1}2, \tau_{W,1}=k+j\}$ 
has probability $P(x,I_j)$ for $j=1,\ldots,l-1$
and $\{X_k\ge \frac{b+1}2, \tau_{W,1}=k+l\}$ has probability $P(x,I_l^-)$. We therefore have
$\mb P(Y_\rho=n, X_k>\tfrac{b+1}2|X_{k-1}=x)$ is
\begin{equation*}
\begin{cases}
\sum_{j=1}^{l-1}2^{-(n-k-j+1)}P(x,I_j)+
2^{-(n-k-l+1)}P(x,I_l^-)&\text{if $n\ge k+l$;}\\
\sum_{j=1}^{n-k-1}2^{-(n-k-j+1)}P(x,I_j)
&\text{otherwise.}
\end{cases}
\end{equation*}

Now an application of Bayes' theorem, together with Lemma \ref{Lem:Est1}, 
shows that conditioned on $X_{k-1}=x$, $X_k\ge\frac{b+1}2$ and $Y=n$, the distribution of
$X_k$ on each interval above is uniform up to a multiplicative factor of the fixed constant $K$,
except for $I_l$ in which the density drops off to 0, but has the property that on $I_l^-$,
the density is within a factor of $K$ of that on $I_{l-1}$. 
Since in the next $\tau_{W,1}-k$ steps, the interval is mapped linearly onto $[b,\frac{b+1}2]$, 
the distribution of $X_{\tau_{W,1}}$ is uniform up to a multiplicative factor of $2K$
on $[b,\frac{b+1}2]$. 
\end{proof}

The following lemma states that after each return to the set $W$, at least a
fixed positive proportion of the mass ends up in the subset $H=[\frac 12,\frac{b+1}2)$ of $W$.
Recall that $\tau_{W,i}$ is defined to be the time of the $i$th reentry to $W$.
\begin{lemma}\label{Lem:ReturnstoW}
There exists $q'>0$ such that for every $x\in(b,\frac{b+1}{2})$
\[
\mb P_x\left(\tau_H\le \tau_{W,1}\right)>q'.
\]
\end{lemma}

\begin{proof}
If $x\in [\frac 12,\frac{b+1}2)$, then $\tau_H=0$ and $\tau_{W,1}>0$ so that 
$\mb P_x(\tau_H\le\tau_{W,1})=1$.
If $x<\frac 12$ then by point (i) in Lemma \ref{Lem:Bounds}, 
\[
\mb P_x\left(X_1>\tfrac{b+1}{2}\right)>r.
\]
It then follows from Lemma \ref{Lem:BoundsW} that
$X_{\tau_{W,1}}$ is absolutely continuous with density at least
$rC^{-1}$. In particular, $\mb P_x(X_{\tau_{W,1}}\in H)>\frac {rb}{C}$,
so that $\mb P_x(\tau_H\le \tau_{W,1})\ge \frac{rb}{C}$.
\end{proof}

We now show that the entry time into $H$ has exponential
tails.
\begin{proposition}\label{Prop:StretExp}  Given any probability measure $\mu$
with density $\rho$ satisfying \eqref{Eq:Condrho}, there is $c'>0$
such that for any $t\in \N$,
\[
\mb P_\mu(\tau_H>t)\leq  e^{-c't}.
\]
\end{proposition}
\begin{proof}
First of all define random variables
$Y_i:=\tau_{W,i}-\tau_{W,i-1}$, and the random variable $N$ equal to
the number of reentries to $W$ up to the first entry to $H$ (recalling that $H\subset W$). With these
definitions, the time to enter $H$ is given by
$\tau_H=\sum_{i=1}^NY_i+V$, where $V$ is the number of steps
spent in $[b,\frac 12]$ after the $N$th reentry to $W$ prior to entering $H$ (note that $V$ may be 0 if
the first point of $W$ that the system enters on the $N$th visit is in $H$). 

Notice that if the constant $K$ in Lemma \ref{Lem:InvClassRho} is chosen sufficiently
large, then conclusion \ref{it:induct} of Lemma \ref{Lem:EstReturnW2}
shows that conditional on $\{\tau_{W,1}<\tau_H\}$, 
the density of $X_{\tau_{W,1}}$ on $[b,\frac 12]$ satisfies \eqref{Eq:Condrho}.

Lemma \ref{Lem:EstReturnW2} and the Markov property, together with Strassen's theorem
 \cite[Theorem 2.4]{Liggett},
imply that the sequence of random variables $(Y_i)$ may be coupled with a sequence
$(Z_i)$ of independent identically distributed random variables with exponential tails 
in such a way that $Y_i\le Z_i$ for $i=1,\ldots,N$.
Lemma \ref{Lem:ReturnstoW} implies that $N$ also has exponential tails.

Let $m>\mb EZ$. A Chernoff bound (see \cite[Section 2.7]{Durrett}) shows that 
there exists $\delta>0$ such that 
%
%Applying a Chernoff bound, for any $n\in\N$
%\begin{equation}\label{Eq:ChernoffBound}
%\mb P_\mu\left(\sum_{i=1}^nZ_i>mn\right)\leq e^{-smn}\prod_{i=1}^n\mb E[e^{sZ_i}]
%\end{equation}
%Since $\{Z_i\}$ have exponential tails and uniformly bounded moments
%(varying $i\in\N$), for any $m>\sup_i\mb E[Z_i]$, there are $\delta>0$
%and $s>0$ such that $\mb E[e^{sZ_i}]< e^{s(m-\delta)}$. Equation
%\eqref{Eq:ChernoffBound} implies that
\begin{equation}\label{Eq:ExpBound}
\mb P_\mu\left(\sum_{i=1}^nZ_i>mn\right)\leq e^{-\delta n}.
\end{equation}
The event $\{\tau_H>t\}$ is a subset of $\{V>\frac t2\}\cup \{N > \frac t{2m}\}
\cup \{N\le \frac{t}{2m},\, Z_1+\ldots +Z_N>\frac t2\}$ and this is a subset of $\{V>\frac t2\}\cup
\{N>\frac t{2m}\}
\cup \{Z_1+\ldots +Z_{\floor{\frac{t}{2m}}} > m\floor{\frac{t}{2m}} \}$.
For the first set $\mb P(V>t/2)\le (1-r)^{\floor{t/2}}$ by Lemma 
\ref{Lem:Bounds}. For the second set,
\[
\mb P_\mu\left(N>\tfrac {t}{2m}\right)\leq e^{-c_2 \frac{t}{2m}},
\]
since $N$ has exponential tails, and for the  third
\[
\mb P_\mu\left( Z_1+\ldots +Z_{\floor{\frac{t}{2m}}} > 
m\floor{\tfrac{t}{2m}}\right)\le e^{-\delta \frac{t}{2m}}
\]
from \eqref{Eq:ExpBound}.
\end{proof}

The following lemma proves that Proposition \ref{Prop:StretExp} can be applied to the density of $X_{\tau_W}$.

\begin{lemma}\label{Lem:LeaveC}
There exists $a<1$ and $n_0$ such that
for all $x$ in the petite set $C=[f_\alpha^{-1}b,b)$, one has
$\mb P_x(\tau_W>n)\le a^n$ for all $n\ge n_0$

There exists a $K>0$ such that for all $x$ in the petite set $C$, 
the distribution of $X_{\tau_W}$, the position at the first entrance
to $W$, is absolutely continuous with density satisfying 
\begin{equation}
\rho(y)\le \frac K{(\log \frac b{y-b})^{1+\epsilon}(y-b)}\label{Eq:CondrhoPlus}
\end{equation}
for all $y\in W$ (that is, the density satisfies the condition
\eqref{Eq:Condrho} extended to the full interval $W$).
\end{lemma}

\begin{proof}
Let $d\in (f_\alpha^{-1}b,b)$ be chosen so that $f_\alpha(d)<\frac 12$.
For any $x\in [f_\alpha^{-1}b,d)$, the time to enter $[d,f_\alpha(d)]$ is bounded 
above by a geometric random variable (with parameter not depending on $x$ 
since $P(x,[d,f_\alpha(d)])\ge P(f_\alpha^{-1}b,[d,f_\alpha(d)])>0$ for all $x\in [f_\alpha^{-1}b,d]$).
For any $x\in [d,b)$, the time to enter $[b,f_\alpha(b)]$ is bounded above by another geometric
random variable by an identical argument. On the interval $(\frac{b+1}2,f_\alpha(b)]$, 
the time to enter $W$ is uniformly bounded above.
Summing these contributions gives the required geometric upper bounds on $\tau_W$. 

We look at the distribution of $X_{\tau_{W}}$. If $x\in [f_\alpha^{-1}b,d]$, 
then we study
\begin{equation}\label{Eq:contri1}
\sum_{n=1}^\infty \mathbf 1_{C^c}(P\mathbf 1_{[f_\alpha^{-1}b,d]})^n\delta_x,
\end{equation}
the contribution to the density on $W$ coming from points that stay in $[f_\alpha^{-1}b,d]$
until they enter $C^c$ (necessarily into $W$ since $f_\alpha(d)<\frac 12$). (Note that this contribution
is trivial if $x>d$.)

If $x\in [f_\alpha^{-1}b,d]$, we showed above that the number of steps before leaving the interval
is bounded above by a geometric random variable; and $p_x(y)$ is bounded
above for $(x,y)\in [f_\alpha^{-1}b,d]\times W$. Combining these,
we see that the density of the contribution in \eqref{Eq:contri1} is uniformly bounded above.

For the remaining part of the distribution, we condition on the last point of $[d,b)$ that is visited.
We show that for all $x\in [d,b)$,
conditional on leaving $C$ in a single step, the density of $X_{\tau_W}$ conditioned
on $X_0=x$ and $X_1\ge b$
satisfies a bound of the form \eqref{Eq:CondrhoPlus}.
First note, that for $x\in [d,b)$ and $y\ge \frac{b+1}2$, $p_x(y)$ is uniformly bounded
above, so that when it is mapped under iteration of the second branch back to $W$,
it gives a contribution that is uniformly bounded above (similar to \eqref{Eq:CollapsetoW}).
It therefore suffices to show that for $x\in [d,b]$, $p_x(y|y\in W)$ satisfies a bound
of the form \eqref{Eq:CondrhoPlus} with a $K$ that does not depend on $x$. Since the probability
of entering $W$ in one step is bounded below for these $x$'s, it suffices to show the existence of a 
$K$ such that for all $x\in[d,b]$ and all $y\in W$, $p_x(y)\le K/[(y-b)(\log\frac b{y-b})^{1+\epsilon}]$.
For $x$ in a small interval $[b-\delta,b]$ and $y$ in a small interval $[b,b+\delta]$, one
may check that $p_x(y)$ is increasing in $x$, so that $p_x(y)\le p_b(y)$ as required. If either $x$
or $y$ lies outside this range, there is a uniform upper bound on $p_x(y)$, establishing the required 
inequality.
\end{proof}

\subsection{Step 2}
Finally, we verify condition \eqref{Eq:CondHitPet} of Theorem \ref{Thm:TT}. 
\begin{proposition}\label{prop:Eretpower}
Suppose that $0<\alpha<1$ is the minimum of the support of the measure $\nu$
and let $1<\gamma<\frac 1\alpha$. For every $x$
belonging to the petite set $C=[f_\alpha^{-1}b,b]$,
\[
\sup_{x\in C}\mb E_x\left[\tau_C^{\gamma}\right]<\infty.
\]
\end{proposition}
\begin{proof}
Notice that we can bound the return time to the petite set $\tau_C$ as
\(
\tau_C\le\tau_W+\tau_{W\to H}+\tau_{H\rightarrow C}
\)
where $\tau_W$ is the first time to hit $W$, $\tau_{W\to H}$ is the random time after then 
that it takes to hit $H$ (this may be 0), and $\tau_{H\rightarrow
C}$ is the random time needed to go back to $C$ starting from when $H$ is hit.  It
follows from Minkowski's inequality that if $\sup_{x\in C}\mb E_x[\tau_W^\gamma]$,
$\sup_{x\in C}\mb E_x[\tau_{W\to H}^\gamma]$ and 
$\sup_{x\in C}\mb E_x[\tau_{H\rightarrow C}^\gamma]$ are finite, then also
$\sup_{x\in C}\mb E_x[\tau_C^\gamma]$ is finite.

By Lemma \ref{Lem:LeaveC}, $\sup_{x\in C}\mb E_x[\tau_W^{\gamma}]<\infty$.
To show that $\sup_{x\in C}\mb E_x[\tau_{W\to H}^\gamma]<\infty$, Lemma
\ref{Lem:LeaveC} shows that the distribution of $X_{\tau_W}$ satisfies \eqref{Eq:CondrhoPlus} and
the conclusion follows from Proposition \ref{Prop:StretExp}.

To estimate $\mb E_x\tau_{H\to C}^\gamma$,
we need to control the distribution of $X_{\tau_H}$. 
The system may enter $H$ in a number of ways: from $C$ without previously entering $W$;
by making a number of visits to $W$ and then entering direct from $W$; 
or by making a number of visits to $W$ and then entering from $[\frac{b+1}2,1]$. 
For direct entry from $C$, conditioning on the last point visited in $C$, the density on $H$
is uniformly bounded above. Similarly, conditioning on the last point visited in $W$ before entering $H$
from the right, Lemma \ref{Lem:BoundsW} gives a uniform upper bound on the contribution to the density 
of $X_{\tau_H}$. 
For the points entering from $W$, Lemmas \ref{Lem:BoundsW} and \ref{Lem:LeaveC} ensure that on 
entry to $[b,\frac 12)$, the density satisfies \eqref{Eq:Condrho} (with a fixed $K$). The sum of the
densities prior to exiting $[b,\frac 12)$ is estimated in \eqref{Eq:sumbound} and this gives a bound on the
density of the last position in $W$ before exiting, which is $1/(1-\sigma)$ times the bound in
\eqref{Eq:Condrho}. Then the second part of Lemma \ref{Lem:InvClassRho} gives a uniform upper 
bound on this contribution to the density on $H$ of $X_{\tau_H}$. Taken together, we have
shown that the distribution of $X_{\tau_H}$ (and therefore the density of its image on $[0,b)$)
is absolutely continuous with density bounded above by a 
constant that is independent of $x$. Since for $z\in [0,b)$, $\tau_C\le \tau_{[\frac 12,1]}$,
and Proposition \ref{Prop:As1and2valid} established condition \eqref{Eq:Assum2},
we see $\mb E_x\tau_{H\to C}^\gamma<\infty$ as required.
\end{proof}

\subsection{Step 3}

\begin{proof}[Proof of Theorem \ref{thm:Linftydecorr}]
Let $1<\gamma<\frac 1\alpha$ and let $r(n)=n^{\gamma-1}$, so that 
$r(n)$ is a subgeometric sequence with $\sum_{j=0}^{n-1} r(n)=(1+o(1))\frac{n^{\gamma}}\gamma$.
Now Proposition \ref{prop:Eretpower} shows that the hypotheses of Theorem \ref{Thm:TT} are
satisfied for this sequence $r(n)$, so that 
$$
\int \|P^n(x,\cdot)-\pi\|_\text{TV}\,dm(x)=o(n^{1-\gamma}).
$$
Finally, for $\phi,\psi\in L^\infty([0,1])$,
\begin{align*}
&\left|\int\int \phi(x)\psi(f_{\bo\omega}^n x)\,dm(x)\,d\nu_{\alpha,\epsilon}^\N(\bo\omega)
-
\int\phi(x)\,dm(x)\int\psi(y)\,d\pi(y)\right|\\
&=
\left|\int \phi(x)\left(\int \psi(y)\,P^n(x,dy)-\int \psi(y)\,d\pi(y)\right)\,dm(x)\right|\\
&\le \|\phi\|_\infty \|\psi\|_\infty \int \|P^n(x,\cdot)-\pi\|_\text{TV}\,dm(x)=o(n^{1-\gamma}).
\end{align*}
\end{proof}

\bibliographystyle{amsalpha}
\bibliography{LSVbib}
\newpage

\end{document}